\documentclass[12pt]{amsart}
\usepackage{amsmath,amsthm,amscd,amssymb,amsfonts}
\usepackage{fullpage,colonequals}
\usepackage{mathrsfs, upgreek, mathabx}
\usepackage{stmaryrd}
\usepackage{tikz}
\usepackage[all,cmtip]{xy}
\usepackage{xcolor}
\usepackage{graphicx, color}
\usepackage{mathtools}
\usepackage[hidelinks]{hyperref}
\hypersetup{
  colorlinks,
  linkcolor={red!85!black},
  citecolor={blue!85!black},
  urlcolor={blue!95!black}
}

\usepackage{yfonts}
\usepackage[T1]{fontenc}

\numberwithin{equation}{section} 
\mathtoolsset{showonlyrefs} 
\usepackage{enumitem} 
\setlist[enumerate]{label=$(\roman*)$, ref=$(\roman*)$}

\theoremstyle{plain}
\newtheorem{theoalph}{Theorem}

\newtheorem{theorem}{Theorem}[section]
\newtheorem{proposition}[theorem]{Proposition}

\newtheorem{coro}[theorem]{Corollary}

\newtheorem{lemma}[theorem]{Lemma}

\theoremstyle{remark}
\newtheorem{remark}[theorem]{Remark}

\newtheoremstyle{citing}
{3pt}
{3pt}
{\itshape}
{}
{\bfseries}
{.}
{.5em}
{\thmnote{#3}}

\theoremstyle{citing}

\newcommand{\A}{\mathbb{A}}

\newcommand{\C}{\mathbb{C}}

\newcommand{\E}{\mathbb{E}}
\newcommand{\F}{\mathbb{F}}
\newcommand{\K}{\mathbb{K}}
\newcommand{\Q}{\mathbb{Q}}
\newcommand{\R}{\mathbb{R}}

\newcommand{\Z}{\mathbb{Z}}

\renewcommand{\H}{\mathbb{H}}
\renewcommand{\P}{\mathbb{P}}

\newcommand{\cA}{\mathcal{A}}
\newcommand{\cB}{\mathcal{B}}

\newcommand{\cE}{\mathcal{E}}
\newcommand{\cF}{\mathcal{F}}

\newcommand{\cH}{\mathcal{H}}

\newcommand{\cK}{\mathcal{K}}

\newcommand{\cM}{\mathcal{M}}
\newcommand{\cN}{\mathcal{N}}
\newcommand{\cO}{\mathcal{O}}

\newcommand{\cU}{\mathcal{U}}

\newcommand{\cX}{\mathcal{X}}

\newcommand{\fD}{\mathfrak{D}}

\newcommand{\fO}{\mathfrak{O}}

\newcommand{\sC}{\mathscr{C}}

\newcommand{\sE}{\mathscr{E}}

\newcommand{\sH}{\mathscr{H}}

\newcommand{\sK}{\mathscr{K}}
\newcommand{\sL}{\mathscr{L}}

\newcommand{\hF}{\widehat{F}}
\newcommand{\hG}{\widehat{G}}

\newcommand{\hS}{\widehat{S}}

\newcommand{\halpha}{\widehat{\alpha}}

\newcommand{\hnu}{\widehat{\nu}}

\newcommand{\hphi}{\widehat{\phi}}

\newcommand{\tE}{\widetilde{E}}

\newcommand{\tvarphi}{\widetilde{\varphi}}

\renewcommand{\=}{\colonequals}
\renewcommand{\:}{\colon}
\newcommand{\dd}{\hspace{1pt}\operatorname{d}\hspace{-1pt}}
\newcommand{\ssetminus}{\smallsetminus}

\newcommand{\whf}{\widehat{f}}

\newcommand{\tcF}{\widetilde{\cF}}

\newcommand{\bfB}{\mathbf{B}}
\newcommand{\bfD}{\mathbf{D}}
\newcommand{\bfG}{\mathbf{G}}
\newcommand{\bfH}{\mathbf{H}}
\newcommand{\bfI}{\mathbf{I}}

\newcommand{\bfO}{\mathbf{O}}
\newcommand{\bfR}{\mathbf{R}}

\newcommand{\bfX}{\mathbf{X}}

\DeclareMathOperator{\Aut}{Aut}
\DeclareMathOperator{\Berk}{Berk}
\DeclareMathOperator{\Div}{Div}
\DeclareMathOperator{\End}{End}

\DeclareMathOperator{\Iso}{Iso}

\DeclareMathOperator{\Hom}{Hom}
\DeclareMathOperator{\GL}{GL}
\DeclareMathOperator{\SL}{SL}

\DeclareMathOperator{\supp}{supp}

\newcommand{\CM}{CM}
\DeclareMathOperator{\hyp}{hyp}
\newcommand{\mhyp}{\mu_{\hyp}}
\newcommand{\odelta}{\overline{\delta}}

\newcommand{\bfone}{\mathbf{1}}
\DeclareMathOperator{\ord}{ord}

\newcommand{\Rp}{\R_{> 0}}
\newcommand{\Rzp}{\R_{\ge 0}}
\newcommand{\N}{\Z_{> 0}}
\newcommand{\Nz}{\Z_{\ge 0}}

\newcommand{\Fp}{\F_{p}}

\newcommand{\Fpalg}{\overline{\F}_{p}}

\newcommand{\Qp}{\Q_{p}}

\newcommand{\Qpalg}{\overline{\Q}_{p}}
\newcommand{\OQpalg}{\cO_{\Qpalg}}

\newcommand{\Cp}{\C_{p}}

\newcommand{\Op}{\cO_{p}}

\newcommand{\OK}{\cO_{\cK}}

\newcommand{\Zp}{\Z_{p}}

\DeclareMathOperator{\bad}{bad}
\DeclareMathOperator{\sups}{sups}
\newcommand{\Ell}{Y}
\newcommand{\Bad}{Y_{\bad}(\Cp)}
\newcommand{\Ord}{Y_{\ord}(\Cp)}

\newcommand{\Sups}{Y_{\sups}(\Cp)}

\newcommand{\tSups}{Y_{\sups}(\Fpalg)}

\newcommand{\AKber}{\A^1_{\Berk}}
\DeclareMathOperator{\can}{can}
\newcommand{\xcan}{x_{\can}}

\newcommand{\FE}{\cF_{E}}

\DeclareMathOperator{\nr}{nr}

\renewcommand{\ss}{e}
\newcommand{\sspr}{\ss'}
\newcommand{\Bss}{\bfB_{\ss}}

\newcommand{\Dss}{\bfD_{\ss}}
\newcommand{\Dsspr}{\bfD_{\sspr}}
\newcommand{\hDss}{\widehat{\bfD}_{\ss}}
\newcommand{\hDsspr}{\widehat{\bfD}_{\sspr}}
\newcommand{\Fss}{\cF_{\ss}}
\newcommand{\Fsspr}{\cF_{\sspr}}
\newcommand{\Gss}{\bfG_{\ss}}

\newcommand{\Piss}{\Pi_{\ss}}
\newcommand{\Pisspr}{\Pi_{\sspr}}
\newcommand{\Rss}{\bfR_{\ss}}

\newcommand{\Xss}{\bfX_{\ss}}
\newcommand{\Xsspr}{\bfX_{\sspr}}

\newcommand{\Nr}{\mathbf{Nr}}
\newcommand{\NE}{\Nr_{E}}

\newcommand{\coset}{\mathfrak{N}}
\DeclareMathOperator{\Orb}{Orb}
\newcommand{\corbit}{\Orb_{\coset}(E)}
\newcommand{\corbitc}{\overline{\corbit}}
\newcommand{\OrbE}{\Orb_{\NE}(E)}
\newcommand{\OrbEc}{\overline{\OrbE}}
\DeclareMathOperator{\Ev}{Ev}
\newcommand{\bT}{\overline{T}}
\newcommand{\on}[1]{\Vert #1 \Vert_0}
\newcommand{\norm}[1]{\Vert #1 \Vert}

\newcommand{\muE}{\mu_{\NE}^E}

\newcommand{\rfk}{\Bbbk_0}

\newcommand{\piem}{\pi_0}

\DeclareMathOperator{\PGL}{PGL}

\newcommand{\Phiss}{\Phi_{\ss}}

\newcommand{\dconst}{W_{\delta,E}}
\newcommand{\zeromean}{W_{\delta,E,0}}

\newcommand{\uC}{\underline{C}}
\newcommand{\uE}{\underline{E}}
\newcommand{\ux}{\underline{x}}
\newcommand{\hcX}{\widehat{\cX}}
\newcommand{\ellc}{\overline{\langle \ell \rangle}}
\newcommand{\Hss}{\bfH_{\ss}}
\DeclareMathOperator{\ev}{ev}

\usepackage{microtype}
\begin{document}


\title{Statistical properties of Hecke correspondences}

\author{Sebasti\'an Herrero}
\address{Universidad de Santiago de Chile, Dept.~de Matem\'atica y Ciencia de la Computaci\'on, Av.~Libertador Bernardo O'Higgins 3363, Santiago, Chile}
\email{sebastian.herrero.m@gmail.com}

\author{Ricardo Menares}
\address{
  Pontificia Universidad Cat\'olica de Chile, Facultad de Matem\'aticas, Vicu\~na Mackenna 4860, Santiago, Chile.}
\email{rmenares.v@gmail.com}

\author{Juan Rivera-Letelier}
\address{Department of Mathematics, University of Rochester. Hylan Building, Rochester, NY~14627, U.S.A.}
\email{riveraletelier@gmail.com}
\urladdr{\url{http://rivera-letelier.org/}}

\begin{abstract}
  This paper studies the statistical properties of the dynamical system generated by a Hecke correspondence on the modular curve over~$\C$ and, for every prime number~$p$, over~$\Cp$.
  Over~$\C$, it proves that the equidistribution to the hyperbolic measure established by Clozel and Otal occurs at an exponential rate.
  Moreover, it determines the sharp rate, assuming an affirmative solution to the Ramanujan--Petersson conjecture.
  Over~$\Cp$, two distinct types of behavior arise.
  In the first, every orbit converges towards the Gauss point in the Berkovich affine line, and this paper establishes the sharp exponential convergence rate.
  In the second, a form of unique ergodicity holds on each orbit closure, and this paper establishes a spectral gap property and, as a consequence, a central limit theorem.
  This complements the central limit theorem proved by Cantat and Le~Borgne over~$\C$.
  Finally, the paper extends and strengthens the results of Goren and Kassaei on the associated random walks.
\end{abstract}

\maketitle
\setcounter{tocdepth}{1}
\tableofcontents

\section{Introduction}

Let~$\K$ be an algebraically closed field of characteristic zero endowed with a norm.
Denote by~$\Ell(\K)$ the moduli space of elliptic curves over~$\K$.
It is the space of all isomorphism classes of elliptic curves over~$\K$, for isomorphisms defined over~$\K$.
Let~$\Div(\Ell(\K))$ denote the group of divisors on~$\Ell(\K)$, \emph{i.e.}, the free Abelian group generated by the points of~$\Ell(\K)$.
For each~$\ell$ in~$\N$, the \emph{$\ell$-th Hecke correspondence} is the linear map
\begin{equation}
  \label{eq:1}
  T_\ell \: \Div (\Ell(\K)) \rightarrow \Div(\Ell(\K)),
\end{equation}
defined for~$E$ in~$\Ell(\K)$ by
\begin{equation}
  \label{eq:2}
  T_\ell(E)
  \=
  \sum_{C \le E \textrm{ of order } \ell} E/C,
\end{equation}
where the sum runs over all subgroups~$C$ of~$E$ of order~$\ell$.

For each prime number~$p$, let~$(\Cp, | \cdot |_p)$ be a completion of an algebraic closure of the ﬁeld of $p$\nobreakdash-adic numbers~$\Qp$.

Given an integer~$\ell$ satisfying ${\ell \ge 2}$, this paper studies the statistical properties of the dynamical system generated by~$T_\ell$ over~$\C$ and, for every prime number~$p$, over~$\Cp$.
Over~$\C$, it proves that the equidistribution to the hyperbolic measure established by Clozel and Otal \cite{CO01} occurs at an exponential rate (Theorem~\ref{t:equid_complex}
and Corollary~\ref{c:equid_complex}).
Moreover, it determines the sharp rate, assuming an affirmative solution to the Ramanujan--Petersson conjecture (paragraph following Corollary~\ref{c:equid_complex_bis}).
Over~$\Cp$, two distinct types of behavior arise.
In the first, every orbit convergences towards the Gauss point in the Berkovich affine line (Theorem~\ref{t:Equidistribucion}$(i)$ and \eqref{eq:sin-tasa} in Corollary~\ref{c:Equidistribucion}), and this paper establishes the sharp exponential convergence rate (Proposition~\ref{p:canonical_sharpness}).
In the second, a form of unique ergodicity holds on each orbit closure.
The limit measure is homogeneous, and the dynamics acts with a spectral gap on suitable spaces of locally constant functions (Theorem~\ref{t:hecke-dynamics_sups}$(i)$).
Consequences include uniform convergence to the limit (Theorem~\ref{t:Equidistribucion}$(ii)$ and \eqref{eq:sin-tasa_bis} in Corollary~\ref{c:Equidistribucion}) and central limit theorems (Theorem~\ref{t:TLC}, Corollary~\ref{c:TLC}, and Remark~\ref{rmk:case_ell_not_in_NE}).
This complements the central limit theorem proved by Cantat and Le~Borgne over~$\C$ \cite{CanlBo05}.
Finally, the paper extends and strengthens the results of Goren and Kassaei in~\cite{GorKas} on the associated random walks, identifying the limit measures and showing their independence of the generating set (Theorem~\ref{t:random_walks}).

For a given~$E$ in~$\Ell(\K)$, previous works have determined the asymptotic distribution of the sequence of orbits~$(T_\ell(E))_{\ell=1}^\infty$.
For~$\K=\C$, see \cite{COU,CloUll04,EskOh06}, as well as~\cite[``Theorem~5.2'']{BurSar91}.
For each prime number~$p$ and~$\K=\Cp$, see~\cite{HerMenRivI, HerMenRivII,Richard}.

\subsection{Accumulation measures}
\label{ss:Equidistribucion}
The \emph{degree} and \emph{support} of a divisor~$\fD$ in~$\Div(\Ell(\K))$, given by ${\fD = \sum_{E\in \Ell(\K)} n_EE}$, are defined by
\begin{equation}
  \label{eq:3}
  \deg(\fD)
  \=
  \sum_{E\in \Ell(\K)} n_E \text{ and } \supp(\fD) \= \{E\in \Ell(\K):n_E\neq 0\},
\end{equation}
respectively.
When ${\deg(\fD) \ge 1}$ and, for every~$E$ in~$\Ell(\K)$, the inequality ${n_E\ge 0}$ holds, denote by~$\odelta_\fD$ the Borel probability measure on~$\Ell(\K)$ defined by
\begin{equation}
  \label{eq:4}
  \odelta_\fD
  \=
  \frac{1}{\deg(\fD)} \sum_{E\in \Ell(\K)}n_E\delta_E,
\end{equation}
where~$\delta_E$ denotes the Dirac measure on~$\Ell(\K)$ at~$E$.

Fix an integer~$\ell$ satisfying ${\ell \ge 2}$.
Putting
\begin{equation}
  \label{eq:5}
  \sigma(\ell)
  \=
  \sum_{d\in \N, d|\ell}d
  \text{ and }
  \bT_\ell
  \=
  \frac{1}{\sigma(\ell)}T_\ell,
\end{equation}
the equality
\begin{equation}
  \label{eq:deg_of_T_ell}
  \deg (T_\ell(E))
  =
  \sigma(\ell)
\end{equation}
holds, and the normalized correspondence~$\bT_\ell$ acts on functions ${f \: \Ell(\K) \to \R}$ via the formula
\begin{equation}
  \label{eq:6}
  (\bT_\ell f)(E)
  \=
  \int f \dd \odelta_{T_\ell(E)}.
\end{equation}

In the complex setting, this paper studies the action of the normalized $\ell$-th Hecke correspondence~$\bT_\ell$ on complex valued functions defined on~$\Ell(\C)$.
The main result is that~$\bT_\ell$ acts with a spectral gap on the space of functions that are square integrable with respect to the hyperbolic measure (Theorem~\ref{t:equid_complex}).
This property implies a quantitative version of the equidistribution result~\cite[\emph{Th{\'e}or{\`e}me}~1]{CO01} (Corollary~\ref{c:equid_complex}) and the exponential decay of correlations (Corollary~\ref{c:equid_complex_bis}).
See, \emph{e.g.}, the books \cite{Bal00b,HenHer01} for background on the spectral gap property and its further consequences.
The proof of the spectral gap property relies on spectral theory methods, following~\cite{CloUll04}.
In contrast, \cite{CO01} relies on tools from ergodic theory, and the realization of~$T_\ell$ as a modular correspondence in the sense of \cite[Section~1.1]{CU03}.

In contrast to the complex setting, the orbit~$(T_\ell^n(E))_{n=0}^\infty$ need not equidistribute $p$\nobreakdash-adically.
Even when it does, the limit measure depends on~$E$.
The second main result of this paper describes every accumulation measure of every orbit (Theorem~\ref{t:Equidistribucion} and Corollary~\ref{c:Equidistribucion}).
These accumulation measures depend on both the reduction type of the starting point~$E$ and~$p$\nobreakdash-adic properties of the index~$\ell$ of the Hecke correspondence~$T_\ell$.
The proof relies on results from~\cite{HerMenRivI,HerMenRivII} and, in the case where~$E$ has supersingular reduction and ${p \nmid \ell}$, it establishes a spectral gap property for the action of~$\bT_{\ell}$ on certain spaces of locally constant functions (Theorem~\ref{t:hecke-dynamics_sups}$(i)$).
The method of \cite{CO01,COU,EskOh06} breaks down over~$\Cp$ because~$\Ell(\Cp)$ is not uniformized by a Lie group.
See \cite[Remark~1.1]{HerMenRivI}.

\subsubsection*{Complex setting}
Put ${\H \= \{z\in \C \: \Im(z)>0\}}$ and consider the action of~$\SL_2(\Z)$ on~$\H$ given by M{\"o}bius transformations.
There is a natural identification of~$\Ell(\C)$ with~$\SL_2(\Z) \backslash\H$.
The hyperbolic measure~$\frac{3}{\pi} \frac{\dd x \dd y}{y^2}$ on~$\H$ induces a probability measure~$\mhyp$ on~$\Ell(\C)$.
The \emph{hyperbolic Laplace operator~$\Delta$} acting on functions defined on~$\Ell(\C)$, is the operator induced by the operator~${y^2\left(\frac{\partial^2}{\partial x^2} + \frac{\partial^2}{\partial y^2} \right)}$ acting on functions defined on~$\H$.
Denote by~$D(\Ell(\C))$ the space of bounded, $C^\infty$ functions ${f \: \Ell(\C) \to \R}$ such that $\Delta(f)$ is also bounded.
Denote by~$\| \cdot \|_2$ the norm of the complex Hilbert space~$L^2(\Ell(\C),\mhyp)$ and by~$L^2_0(\Ell(\C),\mhyp)$ the subspace of~$L^2(\Ell(\C),\mhyp)$ of functions~$f$ satisfying ${\int f \dd \mhyp = 0}$.

For each~$\ell$ in~$\N$, the correspondence~$\bT_\ell$ acts on~$L^2(\Ell(\C),\mhyp)$ as a bounded and self-adjoint operator fixing each constant function, and mapping~$L^2_0(\Ell(\C),\mhyp)$ into itself.
Denote by~$\on{\bT_\ell}$ the operator norm of the restriction of~$\bT_\ell$ to~$L^2_0(\Ell(\C),\mhyp)$.
Moreover, denote by~$d(\ell)$ the number of divisors of~$\ell$ and put
\begin{equation}
  \label{eq:rho}
  \varrho(\ell)
  \=
  \frac{\sqrt{\ell} d(\ell)}{\sigma(\ell)},
\end{equation}
where~$\sigma(\ell)$ is defined in~\eqref{eq:5}.
The arithmetic function~$\varrho$ so defined is multiplicative and, for every~$\ell$ satisfying ${\ell \ge 2}$, the inequality~${\varrho(\ell) < 1}$ holds.

The following is a version of \cite[\emph{Th{\'e}or{\`e}me}~2.1]{CloUll04} for iterates of a single Hecke correspondence.

\begin{theoalph}
  \label{t:equid_complex}
  For every integer~$\ell$ satisfying ${\ell \ge 2}$,
  \begin{equation}
    \label{eq:7}
    \varrho(\ell)
    \le
    \on{\bT_\ell}
    \le
    \varrho(\ell) \ell^{\frac{7}{64}}
    \text{ and }
    \on{\bT_\ell}
    <
    1.
  \end{equation}
  Moreover, for all~$f$ in~$D(\Ell(\C))$ and compact subset~$K$ of~$\Ell(\C)$, there exists~$C_{f,K}$ in~$\Rzp$ such that, for every~$n$ in~$\Nz$,
  \begin{equation}
    \label{eq:8}
    \sup_{z \in K} \left| \bT_\ell^n f(z) - \int f \dd \mhyp\right|
    \le
    C_{f,K} \cdot \on{\bT_\ell}^n.
  \end{equation}
\end{theoalph}

The inequality ${\on{\bT_\ell} < 1}$ in~\eqref{eq:7} combined with the last assertion of Theorem~\ref{t:equid_complex} yield the following quantitative version of the weak convergence of~$(\odelta_{T_\ell^n}(E))_{n = 0}^\infty$ to~$\mhyp$ given by \cite[\emph{Th{\'e}or{\`e}me}~1]{CO01}.

\begin{coro}
  \label{c:equid_complex}
  For all bounded and continuous function ${f \: \Ell(\C) \to \R}$ and~$z$ in~$\Ell(\C)$, the sequence~$(\bT_\ell^n f (z))_{n = 1}^\infty$ converges exponentially fast to~$\int f \dd \mhyp$.
\end{coro}

To state the following corollary of Theorem~\ref{t:equid_complex}, put, for all~$n$ in~$\N$ and~$f$ and~$g$ in~$L^2(\Ell(\C),\mhyp)$,
\begin{equation}
  \label{eq:9}
  \sC_n(f, g)
  \=
  \int (\bT_\ell^n f) g \dd \mhyp - \int f \dd \mhyp \int g \dd \mhyp.
\end{equation}

\begin{coro}
  \label{c:equid_complex_bis}
  For every integer~$\ell$ satisfying ${\ell \ge 2}$, the correspondence~$\bT_\ell$ is exponentially mixing: For all~$f$ and~$g$ in~$L^2(\Ell(\C),\mhyp)$, the sequence of correlations~$(|\sC_n(f, g)|)_{n = 1}^\infty$ decays exponentially.
\end{coro}

The inequality~$\on{\bT_\ell} \le \varrho(\ell) \ell^{\frac{7}{64}}$ in~\eqref{eq:7} provides an explicit version of \cite[Theorem~1.1]{COU} in the case where ${G = \SL_2}$ and ${\Gamma = \SL_2(\Z)}$.
The proof of this explicit inequality relies on a bound for Hecke eigenvalues of Maass cusp forms due to Kim and Sarnak \cite[Appendix~2]{Kim03}, and it is analogous to the proof of a similar estimate in the proof of \cite[\emph{Th{\'e}or{\`e}me}~2.1(a)]{CloUll04}.
This upper-bound of~$\on{\bT_\ell}$ does not imply ${\on{\bT_\ell} < 1}$ directly for ${\ell = 2}$ because ${\varrho(2) 2^{\frac{7}{64}} > 1}$.
A positive solution to the Ramanujan--Petersson conjecture would yield ${\on{\bT_\ell} = \varrho(\ell)}$, see, \emph{e.g.}, \cite[Section~8.5]{Iwa02}, which directly implies ${\on{\bT_\ell} < 1}$.

\subsubsection*{$p$-Adic setting}
Fix a prime number~$p$ and an algebraic closure~$\Fpalg$ of~$\Fp$.
Identify~$\Ell(\Cp)$ with~$\Cp$ via the $j$-invariant map and consider~$\Ell(\Cp)$ as a subspace of the Berkovich affine line~$\AKber$ over~$\Cp$.
Denote by~$\xcan$ the \emph{canonical} or \emph{Gauss point} of~$\AKber$.
The Dirac measure~$\delta_{\xcan}$ is relevant for Theorem~\ref{t:Equidistribucion} below.
See Section~\ref{s:Berkovich} for details.

Recall that an elliptic curve class~$E$ in~$\Ell(\Cp)$ can have \emph{bad}, \emph{ordinary}, or \emph{supersingular reduction}, thus defining a partition of the moduli space~$\Ell(\Cp)$ into three pairwise disjoint subsets: The \emph{bad}, \emph{ordinary}, and \emph{supersingular reduction loci}, denoted by $\Bad$, $\Ord$ and $\Sups$, respectively.
To each~$E$ in~$\Sups$, Section~\ref{ss:homogeneo} attaches a subgroup~$\NE$ of~$\Zp^\times$ containing $(\Zp^\times)^2$ and, to each coset~$\coset$ in~$\Zp^\times/\NE$, a probability measure~$\mu_{\coset}^E$.
This measure is a finite convex combination of projections of Haar measures on certain~$p$\nobreakdash-adic compact Lie groups attached to supersingular elliptic curves over~$\Fpalg$, see~\eqref{eq:def_mu_coset}.
The support of~$\mu_{\coset}^E$ is the closure in~$\Ell(\Cp)$ of the \emph{partial Hecke orbit}~$\corbit$, defined by
\begin{equation}
  \label{eq:10}
  \corbit
  \=
  \bigcup_{n \in \coset \cap \N} \supp(T_n(E)).
\end{equation}
For every pair of distinct cosets~$\coset$ and~$\coset'$ in~$\Zp^\times/\NE$, the supports of the measures~$\mu_{\coset}^E$ and~$\mu_{\coset'}^E$ are disjoint sets and, in particular, $\mu_{\coset}^E$ and~$\mu_{\coset'}^E$ are mutually singular.

\begin{theoalph}
  \label{t:Equidistribucion}
  Let~$p$ be a prime number and~$\ell$ an integer satisfying ${\ell \ge 2}$.
  Then, the following assertions hold.
  \begin{enumerate}
  \item[$(i)$]
    Let~$\cU$ be a neighborhood of~$\xcan$ in~$\AKber$ and~$E$ in~$\Ell(\Cp)$.
    Suppose~$E$ belongs to ${\Bad \cup \Ord}$ or ${p \mid \ell}$.
    Then, there exists~$C$ in~$\R_{> 0}$ such that, for every~$n$ in~$\N$,
    \begin{equation}
      \label{eq:con-tasa}
      \odelta_{T^n_\ell(E)}(\cU)
      \ge
      1 - C
      \begin{cases}
        \varrho(\ell)^n
        & \text{if } E \in \Bad \cup \Ord;
        \\
        \varrho(|\ell|_p^{-1})^n
        & \text{if } E \in \Sups.
      \end{cases}
    \end{equation}
  \item[$(ii)$]
    Let~$E$ be in~$\Sups$ and suppose ${p \nmid \ell}$.
    Then, for every continuous function ${f \: \Ell(\Cp) \to \R}$, the sequences~$(\bT_\ell^{2n}(f))_{n = 1}^\infty$ and~$(\bT_\ell^{2n + 1}(f))_{n = 1}^\infty$ converge uniformly on~$\OrbEc$ to the constant function equal to~$\int f \dd \muE$ and~$\int f \dd \mu_{\ell \NE}^E$, respectively.
  \end{enumerate}
\end{theoalph}

The following corollary is a direct consequence of Theorem~\ref{t:Equidistribucion} and \cite[Proposition~6.9$(ii)$]{HerMenRivII}.

\begin{coro}
  \label{c:Equidistribucion}
  Let~$p$ be a prime number, $\ell$ an integer satisfying ${\ell \ge 2}$, and~$E$ in~$\Ell(\Cp)$.
  If~$E$ belongs to~$\Bad \cup \Ord$ or ${p \mid \ell}$, then the following weak convergence of measures holds,
  \begin{equation}
    \label{eq:sin-tasa}
    \odelta_{T_\ell^n(E)} \rightarrow \delta_{\xcan} \text{ as } n \rightarrow \infty.
  \end{equation}
  If~$E$ belongs to~$\Sups$ and ${p \nmid \ell}$, then the following weak convergence of measures hold,
  \begin{equation}
    \label{eq:sin-tasa_bis}
    \odelta_{T_\ell^{2n}(E)} \to \muE
    \text{ and }
    \odelta_{T_\ell^{2n + 1}(E)} \to \mu_{\ell \NE}^E
    \text{ as }
    n \to \infty.
  \end{equation}
  Moreover, the sequence of measures~$(\odelta_{T_\ell^n(E)})_{n = 0}^\infty$ converges weakly if and only if~$\ell$ belongs to~$\NE$.
\end{coro}

Theorem~\ref{t:canonical_rate} shows that, for a given neighborhood~$\cU$ of~$\xcan$ in~$\AKber$, the constant~$C$ in Theorem~\ref{t:Equidistribucion}$(i)$ depends uniformly on~$E$.
Its proof relies on the results of \cite{HerMenRivI} describing all Hecke orbits converging towards the Gauss point in~$\AKber$.
Moreover, Proposition~\ref{p:canonical_sharpness} proves that the exponential convergence rate in Theorem~\ref{t:Equidistribucion}$(i)$ is sharp.

In the case where~$E$ belongs to~$\Sups$ and ${p \nmid \ell}$, the behavior of~$T_\ell$ shown in Theorem~\ref{t:Equidistribucion}$(ii)$ and~\eqref{eq:sin-tasa_bis} in Corollary~\ref{c:Equidistribucion} is reminiscent of unique ergodicity in dynamical systems.
As in the proof of Theorem~\ref{t:equid_complex}, a key step in the proof of these results is to establish a spectral gap property for the action of~$\bT_\ell$ on certain spaces of locally constant functions.
This result is stated as Theorem~\ref{t:hecke-dynamics_sups}$(i)$ and derived from the results of \cite{HerMenRivII}.

\subsection{Central limit theorems}
\label{ss:TLC}
Fix a prime number~$p$.
For~$\K=\Cp$, a starting point~$E$ in~$\Ell(\Cp)$, and an integer~$\ell$ satisfying ${\ell \ge 2}$, the following result is a central limit theorem for the forward orbit~$(T_\ell^n(E))_{n = 0}^\infty$ refining Theorem~\ref{t:Equidistribucion} and Corollary~\ref{c:Equidistribucion}, and complementing the analogous results of Cantat and Le~Borgne in~\cite{CanlBo05} for ${\K=\C}$.
When~$E$ belongs to ${\Ord \cup \Bad}$ or ${p \mid \ell}$, the limit measure in Theorem~\ref{t:Equidistribucion}$(i)$ and in~\eqref{eq:sin-tasa} in Corollary~\ref{c:Equidistribucion} is a Dirac measure, and the corresponding central limit theorem is trivial.
Thus, this section restricts to the case where~$E$ belongs to~$\Sups$ and ${p \nmid \ell}$.
It provides a quenched version of the central limit theorem refining the uniform convergence in Theorem~\ref{t:Equidistribucion}$(ii)$, which concerns a Markov process starting at the (arbitrary) point~$E$ of~$\Sups$.
Theorem~\ref{t:TLC} below treats the simpler case where~$\ell$ belongs to~$\NE$.
The annealed, or averaged, version of the central limit theorem, concerning a Markov process with respect to its stationary limit measure, follows as a direct consequence, see Corollary~\ref{c:TLC} below.
For~$\ell$ outside~$\NE$, the quenched and annealed versions of the central limit theorem are treated in Remark~\ref{rmk:case_ell_not_in_NE}.
The precise statements of these results rely on the following notation.

Given~$n$ in~$\N$, the group $Z_0(\Ell(\K)^n)$ of \emph{$0$-dimensional cycles} on $\Ell(\K)^n$ is the free Abelian group generated by the points of~$\Ell(\K)^n$.
So, ${Z_0(\Ell(\K))=\Div(\Ell(\K))}$.
The degree of an element~$\fO$ of~$Z_0(\Ell(\K)^n)$, given by
\begin{equation}
  \label{eq:0-cycle}
  \fO
  =
  \sum_{\uE \in \Ell(\K)^n} m_{\uE}(\uE),
\end{equation}
is the integer defined as
\begin{equation}
  \label{eq:11}
  \deg(\fO)
  \=
  \sum_{\uE \in \Ell(\K)^n} m_{\uE}.
\end{equation}
When ${\deg(\fO) \ge 1}$ and, for every~$\uE$ in~$\Ell(\K)^n$, the inequality ${m_{\uE} \ge 0}$ holds, define the Borel probability measure~$\odelta_{\fO}$ on~$\Ell(\K)^n$ by
\begin{equation}
  \label{eq:12}
  \odelta_{\fO}
  \=
  \frac{1}{\deg(\fO)} \sum_{\uE \in \Ell(\K)^n}m_{\uE} \delta_{\uE}.
\end{equation}

For all~$\ell$ in~$\N$, $n$ in~$\Nz$, and~$E$ in~$\Ell(\K)$, define the element~$\fO_{T_\ell}^n(E)$ of~$Z_0(\Ell(\K)^{n + 1})$ as follows.
If ${n = 0}$, then ${\fO_{T_\ell}^0(E) \= E}$.
If ${n \ge 1}$, then let~$I_{n,\ell}(E)$ denote the collection of all $n$-tuples~$(C_0, C_1, \ldots, C_{n - 1})$ for which there is~$(E_0,E_1,\ldots,E_n)$ in~$\Ell(\K)^{n + 1}$ such that ${E_0 = E}$ and such that, for each~$i$ in~$\{0, \ldots, n\}$, the set~$C_i$ is a subgroup of order~$\ell$ of~$E_i$, and, if ${i \neq n}$, then ${E_{i + 1} = E_i/C_i}$.
The point~$(E_0,E_1,\ldots,E_n)$ of~$Z_0(\Ell(\K)^{n + 1})$ is uniquely determined by~$(C_0,C_1,\ldots,C_n)$.
Denote it by~$\uE(C_0,C_1,\ldots,C_n)$.
Then,
\begin{equation}
  \label{eq:O_T_ell(E)^n_def}
  \fO_{T_\ell}^n(E)
  \=
  \sum_{\uC \in I_{n,\ell}(E)} \uE(\uC).
\end{equation}
The equality~$\deg(\fO_{T_\ell}^{n}(E)) = \sigma(\ell)^n$ holds.

Let~$\ell$ be in~$\N$ and recall from Section~\ref{ss:Equidistribucion} that ${\bT_\ell \= \frac{1}{\sigma(\ell)}T_\ell}$.
For every~$n$ in~$\N$ and ${f \: \Ell(\K) \to \R}$, define ${S_n(f) \: \Ell(\K)^n \to \R}$ as
\begin{equation}
  \label{eq:S_n(f)}
  S_n(f)(E_0,\ldots,E_{n - 1})
  \=
  \sum_{i=0}^{n - 1} f(E_i).
\end{equation}
Then, for every~$E$ in~$\Ell(\K)$,
\begin{equation}
  \label{eq:13}
  (\bT_\ell^n f)(E)
  =
  \int f \dd \odelta_{T_\ell^n(E)}
  \text{ and }
  \int S_n(f) \dd \odelta_{\fO_{T_\ell}^{n - 1}(E)}
  =
  \int f \dd \left( \sum_{k=0}^{n - 1} \odelta_{T_\ell^k(E)} \right).
\end{equation}

Theorem~\ref{t:Equidistribucion}$(ii)$ combined with the second equation in~\eqref{eq:13} yields the following statement, which is analogous to the law of large numbers for an arbitrary starting point in~$\Sups$: For all~$E$ in~$\Sups$, integer~$\ell$ in~$\NE$ satisfying ${\ell \ge 2}$, and continuous function ${f \: \Ell(\K) \to \R}$,
\begin{equation}
  \label{eq:14}
  \int \frac{1}{n} S_n(f) \dd \odelta_{\fO_{T_\ell}^{n - 1}(E)} \to \int f \dd \mu^E_{\NE} \text{ as } n \to \infty.
\end{equation}
The following result establishes the corresponding quenched central limit theorem.
For each~$\sigma$ in~$\Rp$, denote by~$\cN_{0,\sigma^2}$ the \emph{centered Gauss normal distribution with variance $\sigma^2$}, meaning the unique Borel probability measure on~$\R$ satisfying
\begin{equation}
  \label{eq:Gauss_dist}
  \cN_{0,\sigma^2}(I)
  =
  \frac{1}{\sqrt{2\pi \sigma^2}} \int_I \exp \left( -\frac{t^2}{2\sigma^2} \right) \dd t \text{ for every interval~$I$ of~$\R$}.
\end{equation}

\begin{theoalph}[Quenched Central Limit Theorem]
  \label{t:TLC}
  Let~$p$ be a prime number, $E$ in~$\Sups$, $\ell$ an integer in~$\NE$ satisfying ${\ell \ge 2}$, and ${f \: \Ell(\Cp) \to \R}$ a function with~$\int f \dd \muE=0$ and whose restriction to~$\OrbEc$ is locally constant.
  Then,
  \begin{equation}
    \label{eq:sigma_f_p-adic}
    \int f^2 \dd \mu^E_{\NE} + 2 \sum_{k=1}^\infty \int f \cdot \bT_\ell^k f \dd \mu^E_{\NE}
  \end{equation}
  converges to a nonnegative real number.
  Denote its square root by~$\sigma_f$.
  If ${\sigma_f > 0}$, then the central limit theorem holds with variance~$\sigma_f^2$ for~$f$ at each starting point in~$\OrbEc$.
  That is, for all~$E'$ in~$\OrbEc$ and interval~$I$ of~$\R$,
  \begin{equation}
    \label{eq:15}
    \odelta_{\fO_{T_\ell}^{n - 1}(E')} \left( \left\{ \uE \in \Ell(\Cp)^n \: \frac{1}{\sqrt{n}}S_n(f)(\uE) \in I \right\} \right) \to \cN_{0,\sigma_f^2}(I) \text{ as } n \to \infty.
  \end{equation}
  Moreover, for fixed~$I$ the convergence is uniform on~$E'$ in~$\OrbEc$.
  If ${\sigma_f = 0}$, then the same properties hold with~$\cN_{0,\sigma_f^2}$ replaced~$\delta_0$.
\end{theoalph}

Since~$\mu^E_{\NE}$ is a probability measure, the following corollary is a direct consequence of the uniform convergence in Theorem~\ref{t:TLC}.

\begin{coro}[Annealed Central Limit Theorem]
  \label{c:TLC}
  Let~$p$, $E$, ${\ell}$, ${f}$, and~$\sigma_f$ be as in Theorem~\ref{t:TLC}.
  If ${\sigma_f > 0}$, then the central limit theorem with variance $\sigma_f^2$ and initial distribution $\mu^E_{\NE}$ holds for~$f$.
  This is, for every interval~$I$ of~$\R$,
  \begin{equation}
    \label{eq:16}
    \int \odelta_{\fO_{T_\ell}^{n - 1}(E')} \left( \left\{ \uE \in \Ell(\Cp)^n \: \frac{1}{\sqrt{n}}S_n(f)(\uE) \in I \right\} \right) \dd \mu^E_{\NE}(E') \to \cN_{0,\sigma_f^2}(I) \text{ as } n \to \infty.
  \end{equation}
  If ${\sigma_f = 0}$, then the same properties hold with~$\cN_{0,\sigma_f^2}$ replaced by~$\delta_0$.
\end{coro}

See Remark~\ref{rmk:case_ell_not_in_NE} for versions of Theorem~\ref{t:TLC} and Corollary~\ref{c:TLC} in the case where~$\ell$ is outside~$\NE$ and ${p \nmid \ell}$.

The proof of Theorem~\ref{t:TLC} relies on the spectral gap property given by Theorem~\ref{t:hecke-dynamics_sups}$(i)$ and the central limit theorem for Markov chains established by Derriennic and Lin in~\cite{DL03}.

Using work of Clozel, Oh, and Ullmo \cite{COU}, Cantat and Le~Borgne adapted Gordin's method to establish analogous results for ${\K = \C}$.
See \cite[\emph{Th{\'e}or{\`e}me}~4.3]{CanlBo05} for a statement analogous to Corollary~\ref{c:TLC} for the stationary measure~$\mu_{\hyp}$, and \cite[\emph{Remarque~4.4(2)}]{CanlBo05} for the convergence~\eqref{eq:15} in Theorem~\ref{t:TLC} for almost every~$E'$ with respect to~$\mu_{\hyp}$.
Theorem~\ref{t:TLC} establishes the convergence~\eqref{eq:15} for every~$E'$ in~$\OrbEc$, and not only for~$\mu^E_{\NE}$-almost every~$E'$.
This is because the Hecke correspondence~$T_\ell$ on~$\Sups$ is given locally by isometries (see Lemmas~\ref{l:action_isometries} and~\ref{l:formal_Hecke_formula}).

\subsection{Random walks and stationary measures}
\label{ss:random_walks}
Fix a prime number~$p$ and an integer~$\ell$ satisfying ${\ell \ge 2}$.
For each~$E$ in~$\Ell(\Cp)$, the orbit~$\bigcup_{n=0}^\infty \supp(T_\ell^n(E))$ of~$E$ by~$T_\ell$ carries a structure of directed graph that provides a natural state space for random walks.
Goren and Kassaei used this approach in \cite{GorKas} to analyze the dynamics of~$T_\ell$, assuming that~$\ell$ is a prime number different from~$p$ and~$E$ belongs to the good reduction locus ${\Ell(\Qpalg) \ssetminus \Ell_{\bad}(\Qpalg)}$.
For every~$N$ in~$\N$ coprime to~$p\ell$, \cite{GorKas} includes the modular curve~$\Ell_1(N)$.
This section restricts to ${N = 1}$ for comparison.
When~$E$ belongs to~$\Ell_{\sups}(\Qpalg)$, \cite{GorKas} uses the Gross--Hopkins period map to relate the problem to the study of random walks on~$\P^1(\Qpalg)$, and applies the work of Benoist and Quint~\cite{BQ14}.
Theorem~\ref{t:random_walks}, below, completes the picture emerging from~\cite{GorKas}.
A brief review of the necessary notation and terminology follows.
See Section~\ref{ss:p-adic-case} for details.

Denote by~$\Q_{p^2}$ the unique unramified quadratic extension of~$\Qp$ inside~$\Cp$, $\Z_{p^2}$ its ring of integers, and~$M_2(\Z_{p^2})$ the ring of 2$\times$2\nobreakdash-matrices with coefficients in~$\Z_{p^2}$.
Fix a supersingular elliptic curve~$\ss$ over~$\Fpalg$, denote its formal group by~$\Fss$ and by~$\Dss$ the set of classes~$E$ in~$\Ell(\Cp)$ reducing to~$\ss$, and put ${\Gss \= \Aut_{\Fpalg}(\Fss)}$.
The group~$\Gss$ is the unit group of a $p$\nobreakdash-adic divison quaternion algebra, and the reduced norm~$\nr$ of this algebra induces a group homomorphism ${\Gss \to \Zp^{\times}}$.
Denoting by~$\hDss$ the (rigid analytic) deformation space of~$\Fss$, there is a natural (rigid analytic) map ${\Piss \: \hDss \rightarrow \Dss}$.
The technical advantage of~$\hDss$ over~$\Dss$ is that there is a natural action of~$\Gss$ on~$\hDss$, akin to the action of~$\GL_2^+(\R)$ on~$\H$.
The \emph{Gross--Hopkins period map} ${\Phiss \: \hDss \to \P^1(\Cp)}$ is a rigid analytic map that is equivariant with respect to the action of~$\Gss$ on~$\hDss$ and the action by M{\"o}bius transformations on~$\P^1(\Cp)$ of the image of~$\Gss$ under an explicit embedding ${\End(\Fss) \to M_2(\Z_{p^2})}$ \cite{HopGro94b}.
Identify~$\End(\Fss)$ with its image under the latter map.

Suppose~$\ell$ is not divisible by~$p$, put ${\langle \ell \rangle \= \{ \ell^n \: n \in \N\}}$, denote by~$\ellc$ the closure of~$\langle \ell \rangle$ in~$\Zp$, and put
\begin{equation}
  \label{eq:17}
  \Hss(\ell)
  \=
  \{ g \in \Gss \: \nr(g) \in \ellc\}.
\end{equation}
The sets~$\ellc$ and~$\Hss(\ell)$ are closed subgroups of~$\Zp^{\times}$ and~$\Gss$, respectively.
For all probability measures~$\mu$ on~$\Hss(\ell)$ and~$\nu$ on~$\P^1(\Cp)$, put
\begin{equation}
  \label{eq:18}
  \mu \ast \nu
  \=
  \int g_* \nu \dd \mu(g).
\end{equation}
The measure~$\nu$ is \emph{$\mu$-stationary} if ${\nu = \mu \ast \nu}$ and \emph{$\mu$-ergodic} if it is extremal among the $\mu$\nobreakdash-stationary ones.

The next result concerns random walks on~$\P^1(\Cp)$ generated by transition probabilities on~$\Hss(\ell)$.
It provides an explicit realization of the limit measure in terms of the homogeneous measures~$\mu_{1}^{\ss,\ss}$ and~$\mu_\ell^{\ss,\ss}$ on~$\Gss$, defined in Section~\ref{ss:homogeneo}, via the composition of the period map~$\Phiss$ with an appropriate evaluation map.
The measure~$\muE$ appearing in Theorem~\ref{t:Equidistribucion}$(ii)$ is also defined in terms of these homogeneous measures.
For all probability measure~$\mu$ on~$\Hss(\ell)$ and~$k$ in~$\N$, define~$\mu^{\ast k}$ recursively by ${\mu^{\ast 1} \= \mu}$ and, if ${k \ge 2}$, by ${\mu^{\ast k} \= \int g_* \mu \dd \mu^{\ast (k - 1)}(g)}$.

\begin{theoalph}
  \label{t:random_walks}
  Let~$p$ be a prime number, $\ell$ be an integer satisfying ${\ell \ge 2}$ that is not divisible by~$p$, and~$\ss$ a supersingular elliptic curve over~$\Fpalg$.
  Moreover, let~$x$ be in~$\hDss$, denote by ${\Ev^{x,\ss} \: \Gss \rightarrow \hDss}$ the evaluation map ${g \mapsto g \cdot x}$, and put
  \begin{equation}
    \label{eq:19}
    \nu_x
    \=
    (\Phiss \circ \Ev^{x,\ss})_*\left(\frac{\mu_{1}^{\ss,\ss}+\mu_\ell^{\ss,\ss}}{2} \right).
  \end{equation}
  Then, the following properties hold.
  \begin{enumerate}
  \item
    The measure~$\nu_x$ is $\Hss(\ell)$-invariant and, putting ${E \= \Piss(x)}$, it satisfies
    \begin{equation}
      \label{eq:20}
      \supp(\nu_x)
      =
      \Hss(\ell) \cdot \Phiss(x)
      =
      \Phiss\circ \Piss^{-1} \left(\left(\OrbEc \cup \overline{\Orb_{\ell\NE}(E)} \right) \cap \Dss\right).
    \end{equation}
    Moreover, for every~$x'$ in~$\hDss$ such that~$\Phiss(x')$ belongs to~$\Hss(\ell) \cdot \Phiss(x)$, the equality ${\nu_{x'} = \nu_x}$ holds.
  \item
    Let~$\mu$ be a Borel probability measure on~$\Hss(\ell)$ whose support generates~$\Hss(\ell)$ as a closed subsemigroup of~$\Gss$.
    Then, $\nu_x$ is $\mu$-ergodic and, for every~$z$ in~$\Hss(\ell) \cdot \Phiss(x)$ and $\mu^{\otimes \N}$-almost every sequence~$(b_n)_{n=1}^\infty$ of elements of~$\Hss(\ell)$, the following weak convergence of measures hold,
    \begin{equation}
      \label{eq:21}
      \frac{1}{n} \sum_{k = 1}^n \mu^{\ast k} \ast \delta_z \to \nu_x
      \text{ and }
      \frac{1}{n} \sum_{k=1}^n\delta_{b_k \cdots b_1 z} \to \nu_x
      \text{ as }
      n \to \infty.
    \end{equation}
  \end{enumerate}
\end{theoalph}

Theorem~\ref{t:random_walks} shows that the limiting distribution of the random walk is independent of the transition probability~$\mu$, provided the starting point~$z$ lies in the same orbit under~$\Hss(\ell)$.
This dynamical behavior of~$T_\ell$ is reminiscent of unique ergodicity in dynamical systems and is consistent with Theorem~\ref{t:Equidistribucion}$(ii)$.
Theorem~\ref{t:random_walks} is a consequence of the work of Benoist and Quint in~\cite{BQ14} and results on partial Hecke orbits and their homogeneous measures from~\cite{HerMenRivII}.

To relate Theorem~\ref{t:random_walks} to the work of Goren and Kassaei in~\cite{GorKas}, let~$\sH_\ell^+$ be the monoid defined by\footnote{This notation drops the dependency on~$\ss$ for simplicity.}
\begin{equation}
  \label{eq:monoid_H1+}
  \sH_\ell^+
  \=
  \{\phi \in \End(\ss): \deg(\phi) \in \langle \ell \rangle\}.
\end{equation}
Using the natural identification of~$\End(\ss)$ inside~$\End(\Fss)$, the closure of~$\sH_\ell^+$ in~$\Gss$ equals~$\Hss(\ell)$ \cite[Proposition~5.5.1]{GorKas}.
Given a finite subset~$\Gamma$ of~$\sH_\ell^+$ containing the identity and generating the monoid~$\sH_\ell^+$, \cite{GorKas} studies random walks generated by the transition probability~$\mu_\Gamma$ on~$\Hss(\ell)$, defined by
\begin{equation}
  \label{eq:measure_mu_Gamma}
  \mu_\Gamma
  \=
  \frac{1}{\# \Gamma} \sum_{\gamma \in \Gamma} \delta_{\gamma}.
\end{equation}
The hypothesis that~$\Gamma$ generates the monoid~$\sH_\ell^+$ implies that~$\Gamma$ generates~$\Hss(\ell)$ as a closed subsemigroup of~$\Gss$, so the hypotheses of Theorem~\ref{t:random_walks} are satisfied with ${\mu = \mu_\Gamma}$.
This result provides an explicit realization of the limit measures appearing in~\cite[Theorems~5.10.5, 5.10.6, and~5.10.7]{GorKas} in terms of the homogeneous measures~$\mu_{1}^{\ss,\ss}$ and~$\mu_\ell^{\ss,\ss}$, and shows that, although the random walks generated by the transition probability~$\mu_{\Gamma}$ depend \emph{a priori} on the finite generating set~$\Gamma$, the $\mu_{\Gamma}$-ergodic measures on~$\P^1(\Cp)$ do not.
In contrast to~\cite{GorKas}, Theorem~\ref{t:random_walks} allows~$\ell$ to be composite and the generating set~$\Gamma$ of~$\sH_\ell^+$ to be infinite.

\subsection{Strategy and organization}
\label{ss:method}

As explained in the introduction, a key step in the proofs of Theorems~\ref{t:equid_complex}, \ref{t:Equidistribucion}$(ii)$, and~\ref{t:TLC} is to establish the spectral gap property for the action of~$\bT_\ell$ on a suitable function space that depends on the base field~$\K$.
Section~\ref{s:complejo} establishes the spectral gap property in the case where ${\K = \C}$ and proves Theorem~\ref{t:equid_complex} and Corollary~\ref{c:equid_complex_bis}.
The spectral gap property holds on~$L^2(\Ell(\C),\mhyp)$ and its proof, as that of \cite[\emph{Th{\'e}or{\`e}me}~2.1]{CloUll04}, relies on the spectral decomposition of~$L^2(\Ell(\C),\mhyp)$ with respect to the hyperbolic Laplace operator, see, \emph{e.g.}, \cite{Iwa02}.
The proof also relies on elementary properties of Hecke correspondences in Section~\ref{s:Hecke-elementary}.

Section~\ref{s:Berkovich} proves a uniform version of Theorem~\ref{t:Equidistribucion}$(i)$ (Theorem~\ref{t:canonical_rate}) and that the exponential convergence rate is sharp (Proposition~\ref{p:canonical_sharpness}).
The proof of Theorem~\ref{t:canonical_rate} relies on results from~\cite{HerMenRivI} describing all Hecke orbits converging towards the Gauss point in~$\AKber$.
The proof of the sharpness of the exponential convergence rate relies on Proposition~\ref{p:Hecke-rate}, which in turn relies on a large deviation estimate.

Section~\ref{ss:p-adic-case} establishes the spectral gap property in the case where ${\K = \Cp}$ and derives Theorem~\ref{t:Equidistribucion}$(ii)$ from it.
Its proof relies on the work of Gross and Hopkins on deformation spaces of formal modules in~\cite{HopGro94b}, recalled in Section~\ref{ss:deformation-spaces}, and on results on the partial Hecke orbits and their homogeneous measures from \cite{HerMenRivII}, recalled in Section~\ref{ss:homogeneo}.
The proof also relies on elementary properties of Hecke correspondences in Section~\ref{s:Hecke-elementary}.
Section~\ref{s:equid_in_sups_locus} establishes the spectral gap property (Theorem~\ref{t:hecke-dynamics_sups}$(i)$) and deduces Theorem~\ref{t:Equidistribucion}$(ii)$.

Section~\ref{s:TLC} establishes the central limit theorems (Theorem~\ref{t:TLC} and Remark~\ref{rmk:case_ell_not_in_NE}).
The proof relies on the spectral gap property in the case where ${\K = \Cp}$ (Theorem~\ref{t:hecke-dynamics_sups}$(i)$) and the central limit theorem for Markov chains established by Derriennic and Lin in \cite{DL03}.
After recalling the latter in Section~\ref{t:CLT_Markov}, Section~\ref{Markov} introduces a Markov chain associated with~$T_\ell$ and establishes some lemmas.
The proof of Theorem~\ref{t:TLC} and its variants (Remark~\ref{rmk:case_ell_not_in_NE}) occupy Section~\ref{sec:proof_TLC}.

Section~\ref{s:random_walks} proves Theorem~\ref{t:random_walks} by combining the work of Benoist and Quint in~\cite{BQ14} with results on partial Hecke orbits and their homogeneous measures from~\cite{HerMenRivII}.

\subsection{Notation}
Throughout this paper, ${\N \=\{1,2,\ldots\}}$ and ${\Nz \= \N \cup \{0\}}$.
Recall that, for each prime number~$p$, ${\Fpalg}$ denotes an algebraic closure of the field with~$p$ elements~$\Fp$, and~$(\Cp, | \cdot |_p)$ a completion of an algebraic closure of the ﬁeld of $p$\nobreakdash-adic numbers~$\Qp$.
So, ${|p|_p = p^{-1}}$.

\subsection*{Acknowledgments}

S.~Herrero and R.~Menares' research is partially supported by ANID FONDECYT Regular grant~1250734.
The authors thank Eyal Goren for helpful comments on a preliminary version of this paper.

\section{Elementary properties of Hecke correspondences}
\label{s:Hecke-elementary}

This section gathers elementary properties of Hecke correspondences used in the proofs of Theorems~\ref{t:equid_complex} and~\ref{t:Equidistribucion}.

The proofs of Theorems~\ref{t:equid_complex} and~\ref{t:Equidistribucion}$(ii)$ rely on the following proposition.

\begin{proposition}
  \label{p:criterion_convergence_iterates}
  Let~$H$ be a Hilbert space with associated operator norm~$\| \cdot \|_H$ and let~$p$ be a prime number.
  Assume that the $\R$-algebra $\R[\bT_\ell \: \ell \in \N \ssetminus p\N]$ acts on $H$ by bounded, self-adjoint operators.
  Then, the following properties hold:
  \begin{enumerate}
  \item[$(i)$]
    For all~$\ell$ in~$\N \ssetminus p\N$ and~$n$ in~$\N$, the equality ${\|\bT_\ell^n\|_H=\|\bT_\ell\|_H^n}$ holds.
  \item[$(ii)$]
    If $\left\|\bT_\ell\right\|_H\to 0$ as ${\ell \to \infty}$ with~$\ell$ in~$\N \ssetminus p\N$, then, for every~$\ell$ in~$\N \ssetminus p\N$ with $\ell\ge 2$, the inequality ${\left\| \bT_\ell\right\|_H < 1}$ holds.
  \end{enumerate}
\end{proposition}

The proof of the Proposition~\ref{p:criterion_convergence_iterates} relies on the following lemma.
Recall that, for all~$m$ and~$m'$ in~$\N$,
\begin{equation}
  \label{eq:general}
  T_m \circ T_{m'}
  =
  \sum_{\substack{d \in \N \\ d \mid \gcd(m, m')}} d \cdot T_{\frac{mm'}{d^2}},
\end{equation}
see, \emph{e.g.}, \cite[p.~50]{DiaIm95}.
In particular,
\begin{equation}
  \label{eq:coprime}
  T_m \circ T_{m'} = T_{mm'} \text{ if } \gcd(m, m')=1.
\end{equation}

\begin{lemma}
  \label{l:combinatoria}
  For all integers~$\ell$ and~$n$ satisfying $\ell\ge 2$ and $n\ge 1$, there exist~$s$ in~$\N$, $\lambda_1$, \ldots, $\lambda_s$ in~$]0,1[$, and $m_2$, \ldots, $m_s$ in~$\N$ such that
  \begin{equation}
    \label{eq:22}
    \bT_\ell^n
    =
    \lambda_1\bT_{\ell^n} + \lambda_2 \bT_{m_2} + \cdots + \lambda_s\bT_{m_s},
    \lambda_1 + \cdots + \lambda_s
    =
    1,
  \end{equation}
  and, for every~$i$ in~$\{2, \ldots ,s\}$, every prime divisor of~$m_i$ divides~$\ell$.
\end{lemma}

\begin{proof}
  Let~$p$ be a prime number and~$r$ in~$\N$.
  Taking $n=p$ and $n'=p^r$ in~\eqref{eq:general}, yields
  \begin{equation}
    \label{eq:primes}
    \bT_p \circ \bT_{p^r}
    =
    \alpha_{r,p} \bT_{p^{r+1}}+\beta_{r,p} \bT_{p^{r-1}}
  \end{equation}
  with $\alpha_{r,p} \= \frac{\sigma(p^{r+1})}{\sigma(p) \cdot \sigma(p^r)}$ and $\beta_{r,p} \= \frac{p\sigma(p^{r-1})}{\sigma(p) \cdot \sigma(p^r)}$.
  Clearly, ${\alpha_{r,p}+\beta_{r,p} = 1}$.
  Combined with an induction argument, this implies the desired assertion for every~$\ell$ that is a power of~$p$.

  For an arbitrary~$\ell$, write its prime factorization ${\ell = p_1^{r_1} p_2^{r_2} \cdots p_k^{r_k}}$.
  The multiplicativity of~$\sigma$ and~\eqref{eq:coprime} imply, for every~$n$ in~$\N$,
  \begin{equation}
    \label{eq:23}
    \bT_\ell^n
    =
    \bT_{p_1^{r_1}}^n \circ \bT_{p_2^{r_2}}^n \circ \cdots \circ \bT_{p_k^{r_k}}^n.
  \end{equation}
  The general case follows from this identity, the previous case, and~\eqref{eq:coprime}.
\end{proof}

\begin{proof}[Proof of Proposition~\ref{p:criterion_convergence_iterates}]
  Item~$(i)$ is immediate if ${\|\bT_\ell\|_H = 0}$.
  Suppose ${\|\bT_\ell\|_H > 0}$ and denote norm and inner product of~$H$ by~$| \cdot |_H$ and~$\langle \cdot, \cdot \rangle_H$, respectively.
  Since~$\bT_\ell$ acts on $H$ as a bounded and self-adjoint operator, for every~$v$ in~$H$,
  \begin{equation}
    \label{eq:24}
    | \bT_\ell v |_H^2
    =
    \langle \bT_\ell^2 v, v \rangle_H
    \le
    | \bT_\ell^2 v |_H |v|_H.
  \end{equation}
  This yields, ${\| \bT_\ell^2 \|_H \ge \| \bT_\ell \|_H^2}$ and, by the sub-multiplicativity of~$\| \cdot \|_H$, the equality ${\| \bT_\ell^2 \|_H = \| \bT_\ell \|_H^2}$.
  By induction, for every~$n$ in~$\N$, the equality ${\| \bT_\ell^{2n} \|_H = \| \bT_\ell \|_H^{2n}}$ holds.
  Combined with the sub-multiplicativity of~$\| \cdot \|_H$, this implies
  \begin{equation}
    \label{eq:25}
    \|\bT_\ell\|_H^{2n+1}
    \le
    \frac{\|\bT_\ell^{2n+2} \|_H}{\|\bT_\ell\|_H}
    \le
    \frac{ \|\bT_\ell\|_H \cdot \|\bT_\ell^{2n+1} \|_H}{\|\bT_\ell\|_H}
    =
    \|\bT_\ell^{2n+1} \|_H
    \le
    \|\bT_\ell\|_H^{2n+1},
  \end{equation}
  and ${\| \bT_\ell^{2n + 1} \|_H = \| \bT_\ell \|_H^{2n + 1}}$.

  To prove item~$(ii)$, suppose $\left\|\bT_m\right\|_H\to 0$ as ${m \to \infty}$ with~$m$ in~$\N \ssetminus p\N$.
  Put
  \begin{equation}
    \label{eq:26}
    M
    \=
    \sup \{ \| \bT_m \| \: m \in \N \ssetminus p\N \},
  \end{equation}
  and note that ${M < +\infty}$.
  The first step is to prove ${M \le 1}$.
  Let~$m$ be in ${\N \ssetminus p\N}$.
  Lemma~\ref{l:combinatoria} implies that, for every~$n$ in~$\N$, there are~$s$ in~$\N$, $\lambda_1$, \ldots, $\lambda_s$ in~$]0, 1[$ with $\sum_{i=1}^s\lambda_i=1$, and $m_2$, \ldots, $m_s$ in~$\N \ssetminus p\N$ such that
  \begin{equation}
    \label{eq:recurrencia_TT}
    \bT_m^n
    =
    \lambda_1 \bT_{m^n} + \sum_{i=2}^s\lambda_i \bT_{m_i}.
  \end{equation}
  Then, $\|\bT_m^n\|_H \le M( \lambda_1 + \cdots + \lambda_s) = M$.
  Combined with item~$(i)$, this implies that the sequence $(\|\bT_m\|_H^n)_{n = 1}^\infty$ is bounded and hence ${\|\bT_m\|_H \le 1}$.
  Since this holds for every~$m$ in ${\N \ssetminus p\N}$, it follows that ${M \le 1}$.
  To complete the proof of item~$(ii)$, let~$\ell$ in~$\N \ssetminus p\N$ with ${\ell\ge 2}$ be given.
  By hypothesis, there exists~$N$ in~$\N$ such that ${\| \bT_{\ell^N} \| < \frac{1}{2}}$.
  Lemma~\ref{l:combinatoria} implies that there are $s'$ in~$\N$, $\lambda_1'$, \ldots, $\lambda_{s'}'$ in~$]0, 1[$ with $\sum_{i=1}^{s'} \lambda_i' = 1$, and $m_2$, \ldots, $m_{s'}$ in~$\N \ssetminus p\N$ such that
  \begin{equation}
    \label{eq:27}
    \bT_\ell^N
    =
    \lambda_1' \bT_{\ell^N} + \sum_{i=2}^{s'} \lambda_i' \bT_{m_i}.
  \end{equation}
  Combined with item~$(i)$ and the choice of~$N$, this implies
  \begin{equation}
    \label{eq:28}
    \| \bT_\ell \|_H^N
    =
    \| \bT_\ell^N \|_H
    \le
    \frac{\lambda_1'}{2} + \sum_{i=2}^{s'} \lambda_i'
    =
    1 - \frac{\lambda_1'}{2}
    <
    1
    \text{ and }
    \| \bT_\ell \|_H
    <
    1.
    \qedhere
  \end{equation}
\end{proof}

The proof of Theorem~\ref{t:Equidistribucion}$(i)$ relies on the following proposition.

\begin{proposition}
  \label{p:Hecke-composition}
  Let~$m$ in~$\Nz$ and~$n$ in~$\N$ be given.
  Then, there are~$a_{m,n}(0)$, \ldots, $a_{m,n}(\left\lfloor \frac{mn}{2} \right\rfloor)$ in~$\Nz$ satisfying the following properties.
  \begin{enumerate}
  \item[$(i)$]
    For every prime number~$q$,
    \begin{equation}
      \label{id}
      T_{q^m}^n=\sum_{j=0}^{\left\lfloor \frac{mn}{2} \right\rfloor}a_{m,n}(j)q^jT_{q^{mn-2j}}.
    \end{equation}
    In particular, $\sum\limits_{j=0}^{\left\lfloor \frac{mn}{2} \right\rfloor}a_{m,n}(j)q^j\sigma(q^{mn-2j})=\sigma(q^m)^n$.
  \item[$(ii)$]
    The following equality holds,
    \begin{equation}
      \label{suma total}
      \sum_{j=0}^{\left\lfloor \frac{mn}{2} \right\rfloor}a_{m,n}(j)(mn+1-2j)=(m+1)^n.
    \end{equation}
    In particular, $\sum\limits_{j=0}^{\left\lfloor \frac{mn}{2} \right\rfloor}a_{m,n}(j) \le (m+1)^n.$
  \end{enumerate}
\end{proposition}

The proof of Proposition~\ref{p:canonical_sharpness} relies on the following proposition.

\begin{proposition}
  \label{p:Hecke-rate}
  For all~$j$ and~$m$ in~$\Nz$ and~$n$ in~$\N$ satisfying~$j \le \left\lfloor \frac{mn}{2} \right\rfloor$, let~$a_{m, n}(j)$ be given by Proposition~\ref{p:Hecke-composition}.
  Then, for all~$m$ in~$\N$ and~$q$ in~$\R_{> 1}$,
  \begin{equation}
    \label{eq:29}
    \lim_{n \to \infty} \frac{1}{n} \log \left( \sum_{j=0}^{\left\lfloor \frac{mn}{2} \right\rfloor}a_{m,n}(j) q^j \right)
    =
    \log (q^{\frac{m}{2}}(m + 1)).
  \end{equation}
\end{proposition}

The proof of Propositions~\ref{p:Hecke-composition} and~\ref{p:Hecke-rate} rely on a couple of lemmas.
For every~$n$ in~$\Nz$, define the Laurent polynomial~$L_n(X)$ in~$\Z[X,X^{-1}]$ by
\begin{equation}
  \label{eq:30}
  L_n(X)
  \=
  \sum_{j=0}^n X^{n-2j}.
\end{equation}
A simple computation yields the following formula, valid for all integers~$m$ and~$n$ satisfying $0\le m\le n$,
\begin{equation}
  \label{eq:XmLn}
  (X^{-m}+X^m)L_n(X)
  =
  L_{n+m}(X)+L_{n-m}(X).
\end{equation}

\begin{lemma}
  \label{l:mult-Lm}
  For all~$m$ and~$n$ in~$\Nz$,
  $$L_m(X)L_n(X)
  =
  \sum_{j=0}^{\min\{m,n\}}L_{m+n-2j}(X).$$
\end{lemma}

\begin{proof}
  The proof restricts to ${m \le n}$, by symmetry, and it proceeds by induction on~$m$.
  The identity holds trivially when $m=0$.
  When ${m = 1}$ and ${n \ge 1}$, \eqref{eq:XmLn} yields
  \begin{equation}
    \label{eq:31}
    L_1(X)L_n(X)
    =
    (X+X^{-1})L_n(X)=L_{n+1}(X)+L_{n-1}(X).
  \end{equation}
  Hence, the desired identity holds whenever~$m$ belongs to~$\{0,1\}$.
  Suppose that the result holds for a given~$m_0$ in~$\Nz$ and for every integer~$n$ with $n\ge m_0$.
  Then, by~\eqref{eq:XmLn}, for every integer $n$ with $n\ge m_0+2$,
  \begin{equation}
    \label{eq:32}
    \begin{split}
      L_{m_0+2}(X)L_n(X)
      & =
        (X^{-m_0-2}+X^{m_0+2}+L_{m_0}(X))L_n(X)
      \\ & =
           (X^{-m_0-2}+X^{m_0+2})L_n(X)+L_{m_0}(X)L_n(X)
      \\ & =
           L_{n+m_0+2}(X)+L_{n-m_0-2}(X)+\sum_{j=0}^{m_0}L_{m_0+n-2j}
      \\ & =
           \sum_{j=0}^{m_0+2}L_{m_0+2+n-2j}.
    \end{split}
  \end{equation}
  This proves the desired identity with ${m = m_0 + 2}$ and completes the proof of the lemma.
\end{proof}

\begin{lemma}
  \label{l:power-Lm}
  For all~$m$ in~$\Nz$ and~$n$ in~$\N$, there are~$a_{m,n}(0)$, \ldots, $a_{m,n}(\left\lfloor \frac{mn}{2} \right\rfloor)$ in~$\Nz$ satisfying,
  \begin{equation}
    \label{eq:33}
    L_m(X)^n
    =
    \sum_{j=0}^{\left\lfloor \frac{mn}{2} \right\rfloor}a_{m,n}(j)L_{mn-2j}(X).
  \end{equation}
\end{lemma}

\begin{proof}
  For ${n = 1}$, define $a_{m,1}(j)$ as~$1$ if $j=0$, and $0$ if $j>0$.
  Then, \eqref{eq:33} holds trivially for $n=1$.
  Suppose that, for a given~$n$ in~$\N$ and every~$j$ in~$\left\{0,\ldots,\left\lfloor \frac{mn}{2} \right\rfloor \right\}$, there is~$a_{m,n}(j)$ such that~\eqref{eq:33} holds.
  Lemma~\ref{l:mult-Lm} yields,
  \begin{equation}
    \label{eq:34}
    \begin{split}
      L_m(X)^{n+1}
      &=
        L_m(X) L_m(X)^n
      \\ & =
           L_m(X) \sum_{r=0}^{\left\lfloor \frac{mn}{2} \right\rfloor}a_{m,n}(r)L_{mn-2r}(X)
      \\ & =
           \sum_{r=0}^{\left\lfloor \frac{mn}{2} \right\rfloor}a_{m,n}(r) \sum_{k=0}^{\min\{mn-2r,m\}}L_{m(n+1)-2(r+k)}(X)
      \\ & =
           \sum_{j=0}^{\left\lfloor \frac{m(n+1)}{2} \right\rfloor}L_{m(n+1)-2j}(X) \left(\sum^{\left\lfloor \frac{mn}{2} \right\rfloor}_{r=0}a_{m,n}(r)b_{m, n}(j,r) \right),
    \end{split}
  \end{equation}
  where $b_{m, n}(j,r)$ is defined as~$1$ if $0\le j-r\le \min\{mn-2r,m\}$, and~$0$ otherwise.
  This proves~\eqref{eq:33} with~$n$ replaced by~$n+1$, and
  $$a_{m,n+1}(j)
  \=
  \sum^{\left\lfloor \frac{mn}{2} \right\rfloor}_{r=0}a_{m,n}(r)b_{m, n}(j,r).$$
  The lemma follows by induction.
\end{proof}

\begin{remark}
  For every~$n$ in~$\Nz$, let $U_n(X)$ denote the $n$-th Chebyshev polynomial of the second kind.
  These polynomials can be defined recursively by setting ${U_0(X) \= 1}$, ${U_1(X) \= 2X}$, and ${U_{n+1}(X) \= 2XU_n(X) - U_{n-1}(X)}$ for ${n\geq 1}$.
  It is easy to see that, for every~$n$ in~$\Nz$, the equality $L_n(X)=U_n\left(\frac{X+X^{-1}}{2} \right)$ holds.
  Moreover, Lemmas~\ref{l:mult-Lm} and~\ref{l:power-Lm} are still valid when replacing each~$L_n(X)$ with~$U_n(X)$ and can therefore be seen as properties of Chebyshev polynomial of the second kind.
\end{remark}

\begin{proof}[Proof of Proposition~\ref{p:Hecke-composition}]
  To prove item~$(i)$, note that~\eqref{eq:general} implies that, for all~$m$ and~$n$ in~$\Nz$,
  $$T_{q^m} \circ T_{q^n}=\sum_{j=0}^{\min\{m,n\}}q^jT_{q^{m+n-2j}}.$$
  Together with Lemma~\ref{l:mult-Lm}, this implies that the assignation~$L_m(X) \mapsto q^{-m/2}T_{q^m}$ extends to an isomorphism of $\R$-algebras from $\R[L_m \: m\in \Nz]$ to $\R[T_{q^m} \: m\in \Nz]$.
  Hence,~\eqref{id} follows directly from Lemma~\ref{l:power-Lm} by applying this isomorphism.
  Evaluating~\eqref{id} at a point~$E$ of~$\Ell(\K)$, taking degrees on both sides, and using~\eqref{eq:deg_of_T_ell}, yields the last equality in item~$(i)$.

  Evaluating $X=1$ in Lemma~\ref{l:power-Lm} yields~\eqref{suma total}.
  The inequality in item~$(ii)$ follows from the nonnegativity of the coefficients involved.
\end{proof}

\begin{proof}[Proof of Proposition~\ref{p:Hecke-rate}]
  By Proposition~\ref{p:Hecke-composition}($ii$),
  \begin{equation}
    \label{eq:35}
    \sum_{j=0}^{\left\lfloor \frac{mn}{2} \right\rfloor}a_{m,n}(j) q^j
    \le
    q^{\left\lfloor \frac{mn}{2} \right\rfloor} \sum_{j=0}^{\left\lfloor \frac{mn}{2} \right\rfloor}a_{m,n}(j)
    \le
    q^{\left\lfloor \frac{mn}{2} \right\rfloor} (m + 1)^n.
  \end{equation}
  Thus,
  \begin{equation}
    \label{eq:36}
    \limsup_{n \to \infty} \frac{1}{n} \log \left( \sum_{j=0}^{\left\lfloor \frac{mn}{2} \right\rfloor}a_{m,n}(j) q^j \right)
    \le
    \log (q^{\frac{m}{2}} (m + 1)).
  \end{equation}

  Let~$(X_n)_{n = 0}^\infty$ be a sequence of i.i.d. random variables with common distribution
  \begin{equation}
    \label{eq:37}
    \frac{1}{m + 1} \sum_{j = 0}^m \delta_{m - 2j},
  \end{equation}
  and denote by~$\P$ the underlying probability.
  The next step is to prove that, for all~$n$ in~$\N$ and~$k$ in~$\Z$, the coefficient of~$X^k$ in~$L(X)^n$ equals
  \begin{equation}
    \label{eq:38}
    (m + 1)^n \P(X_0 + \cdots + X_{n - 1} = k).
  \end{equation}
  For ${n = 1}$, this follows from the definitions.
  The general case follows by induction, using that, for all~$n$ in~$\N$ and~$k$ in~$\Z$,
  \begin{equation}
    \label{eq:39}
    \P(X_0 + \cdots + X_n = k)
    =
    \frac{1}{m + 1} \sum_{j = 0}^{m} \P(X_0 + \cdots + X_{n - 1} = k - 2j).
  \end{equation}
  It follows that, for every~$n$ in~$\N$, the distribution of ${X_0 + \cdots + X_{n - 1}}$ is supported on ${\{ mn - 2t \: t \in \{0, \ldots, mn \} \}}$, and, by the symmetry of~\eqref{eq:37} and Lemma~\ref{l:power-Lm}, that, for every~$t$ in~$\left\{0, \ldots, \left\lfloor \frac{mn}{2} \right\rfloor \right\}$,
  \begin{multline}
    \label{eq:40}
    \P(X_0 + \cdots + X_{n - 1} = mn - 2t)
    =
    \P(X_0 + \cdots + X_{n - 1} = -(mn - 2t))
    \\ =
    \frac{1}{(m + 1)^n} \sum_{j = 0}^t a_{m, n}(j).
  \end{multline}

  Denote by~$\E[\exp(\lambda X_0)]$ the expectation of the random variable~$\exp(\lambda X_0)$, let~$\Lambda$ be the \emph{logarithmic moment generating function} associated with~\eqref{eq:37}, defined by
  \begin{equation}
    \label{eq:41}
    \Lambda(\lambda)
    \=
    \log \E[\exp(\lambda X_0)],
  \end{equation}
  and let ${\Lambda^* \: \R \to \R \cup \{ +\infty \}}$ be the \emph{Fenchel--Legendre transform of~$\Lambda$}, defined by
  \begin{equation}
    \label{eq:42}
    \Lambda^*(x)
    \=
    \sup_{\lambda \in \R} \{ \lambda x - \Lambda(\lambda) \}.
  \end{equation}
  A direct computation yields
  \begin{equation}
    \label{eq:43}
    \Lambda(\lambda)
    =
    \log \left( \frac{1}{m + 1} \sum_{j = 0}^m \exp(\lambda(m - 2j)) \right).
  \end{equation}
  This implies that~$\Lambda$ and hence~$\Lambda^*$ are even functions, and ${\Lambda'(\R) \supseteq ]-m, m[}$.
  It follows that~$\Lambda$ is strictly convex and continuous on~$]-m, m[$ and satisfies ${\Lambda^*(0) = 0}$, see, \emph{e.g.}, \cite[Lemma~2.2.5 and Exercise~2.2.24]{DemZei10}.
  Thus, Cram{\'e}r's theorem yields, for every~$\varepsilon$ in~$]0, m[$,
  \begin{equation}
    \label{eq:44}
    \limsup_{n \to \infty} \frac{1}{n} \log \P(|X_0 + \cdots + X_{n - 1}| \ge \varepsilon n)
    \le
    -\Lambda^*(\varepsilon)
    <
    0,
  \end{equation}
  see, \emph{e.g.}, \cite[Theorem~2.2.3]{DemZei10}.
  Combined with~\eqref{eq:40}, this implies
  \begin{multline}
    \label{eq:45}
    \limsup_{n \to \infty} \frac{1}{n} \log \left( \frac{1}{(m + 1)^n} \sum_{j = 0}^{\left\lfloor \frac{mn}{2} \right\rfloor - \left\lceil \frac{\varepsilon n}{2} \right\rceil} a_{m, n}(j) \right)
    \\ =
    \limsup_{n \to \infty} \frac{1}{n} \log \P \left( X_0 + \cdots + X_{n - 1} = mn - 2\left\lfloor \frac{mn}{2} \right\rfloor + 2\left\lceil \frac{\varepsilon n}{2} \right\rceil \right)
    \\ \le
    \limsup_{n \to \infty} \frac{1}{n} \log \P(|X_0 + \cdots + X_{n - 1}| \ge \varepsilon n)
    \le
    -\Lambda^*(\varepsilon).
  \end{multline}
  Moreover, fixing~$\varepsilon$, \eqref{eq:44} yields, for every sufficiently large~$n$,
  \begin{equation}
    \label{eq:46}
    \P(|X_0 + \cdots + X_{n - 1}| < \varepsilon n)
    =
    1 - \P(|X_0 + \cdots + X_{n - 1}| \ge \varepsilon n)
    \ge
    1 - \exp(-n\Lambda^*(\varepsilon))
    \ge
    \frac{1}{2}.
  \end{equation}
  Combined with~\eqref{eq:40}, this leads, for every sufficiently large~$n$, to
  \begin{equation}
    \label{eq:47}
    \frac{1}{2}
    \le
    \frac{2}{(m + 1)^n} \sum_{t = \left\lfloor \frac{mn - \varepsilon n}{2} \right\rfloor}^{\left\lfloor \frac{mn}{2} \right\rfloor} \sum_{j = 0}^t a_{m, n}(j)
    \le
    \frac{\varepsilon n + 4}{(m + 1)^n} \sum_{j = 0}^{\left\lfloor \frac{mn}{2} \right\rfloor} a_{m, n}(j).
  \end{equation}
  Combined with~\eqref{eq:45}, this implies
  \begin{multline}
    \label{eq:48}
    \liminf_{n \to \infty} \frac{1}{n} \log \left( \sum_{j = \left\lfloor \frac{mn}{2} \right\rfloor - \left\lceil \frac{\varepsilon n}{2} \right\rceil}^{\left\lfloor \frac{mn}{2} \right\rfloor} a_{m, n}(j) \right)
    \\
    \begin{aligned}
      & =
        \liminf_{n \to \infty} \frac{1}{n} \log \left( \sum_{j = 0}^{\left\lfloor \frac{mn}{2} \right\rfloor} a_{m, n}(j) - \sum_{j = 0}^{\left\lfloor \frac{mn}{2} \right\rfloor - \left\lceil \frac{\varepsilon n}{2} \right\rceil - 1} a_{m, n}(j) \right)
      \\ & \ge
           \liminf_{n \to \infty} \frac{1}{n} \log \left( \frac{(m + 1)^n}{2(\varepsilon n + 4)} - (m + 1)^n \exp(-n\Lambda^*(\varepsilon)) \right)
      \\ & =
           \log (m + 1).
    \end{aligned}
  \end{multline}
  Thus,
  \begin{equation}
    \label{eq:49}
    \begin{split}
      \liminf_{n \to \infty} \frac{1}{n} \log \left( \sum_{j=0}^{\left\lfloor \frac{mn}{2} \right\rfloor}a_{m,n}(j) q^j \right)
      & \ge
        \liminf_{n \to \infty} \frac{1}{n} \log \left( \sum_{j = \left\lfloor \frac{mn}{2} \right\rfloor - \left\lceil \frac{\varepsilon n}{2} \right\rceil}^{\left\lfloor \frac{mn}{2} \right\rfloor} a_{m, n}(j) q^j \right)
      \\ & \ge
           \liminf_{n \to \infty} \frac{1}{n} \log \left( q^{\left\lfloor \frac{mn}{2} \right\rfloor - \left\lceil \frac{\varepsilon n}{2} \right\rceil} \sum_{j = \left\lfloor \frac{mn}{2} \right\rfloor - \left\lceil \frac{\varepsilon n}{2} \right\rceil}^{\left\lfloor \frac{mn}{2} \right\rfloor} a_{m, n}(j) \right)
      \\ & =
           \log (q^{\frac{m - \varepsilon}{2}}) + \liminf_{n \to \infty} \frac{1}{n} \log \left( \sum_{j = \left\lfloor \frac{mn}{2} \right\rfloor - \left\lceil \frac{\varepsilon n}{2} \right\rceil}^{\left\lfloor \frac{mn}{2} \right\rfloor} a_{m, n}(j) \right)
      \\ & \ge
           \log (q^{\frac{m - \varepsilon}{2}} (m + 1)).
    \end{split}
  \end{equation}
  Since this holds for every~$\varepsilon$ in~$]0, m[$, this completes the proof of the proposition.
\end{proof}

\begin{remark}
  Proposition~\ref{p:Hecke-rate} and the key estimate~\eqref{eq:48} in particular can be proved using the asymptotic behavior of the generalized binomial coefficients, see, \emph{e.g.}, \cite{Sac96}, thus avoiding Cram{\'e}r's Theorem.
\end{remark}

\section{The complex setting}
\label{s:complejo}

This section proves Theorem~\ref{t:equid_complex} in Section~\ref{ss:Equidistribucion} and derives Corollary~\ref{c:equid_complex_bis}.
As in the proof of \cite[\emph{Th{\'e}or{\`e}me}~2.1]{CloUll04}, the proof of Theorem~\ref{t:equid_complex} relies on the spectral decomposition of~$L^2(\Ell(\C),\mhyp)$ with respect to the hyperbolic Laplace operator~$\Delta$, see, \emph{e.g.}, \cite[(3.16), p.~61, and Theorems~4.7 and~7.3]{Iwa02}.
More precisely, consider the orthogonal decomposition
\begin{equation}
  \label{eq:50}
  L^2(\Ell(\C),\mhyp)
  =
  \sE_\infty \oplus \bigoplus_{k=0}^\infty\C\varphi_k,
\end{equation}
where~$\sE_\infty$ is the closed subspace of~$L^2_0(\Ell(\C),\mhyp)$ corresponding to the continuous part of the spectrum of~$\Delta$ and~$\left( \varphi_k\right)_{k=0}^\infty$ is a maximal orthogonal system of eigenfunctions of~$\Delta$.
To describe~$\sE_\infty$ more precisely, denote by~$\SL_2(\Z)_\infty$ the stabilizer subgroup of~$\infty$ in~$\SL_2(\Z)$ and, for each~$s$ in~$\C$ satisfying ${\Re(s) > 1}$, by~$E_\infty(\cdot, s)$ the \emph{non-holomorphic Eisenstein series} attached to the cusp~$\infty$ of~$\SL_2(\Z)$ defined, for every~$z$ in~$\H$, by
\begin{equation}
  \label{eq:51}
  E_\infty(z,s)
  \=
  \sum_{\gamma \in \SL_2(\Z)_\infty \backslash \SL_2(\Z)} \Im(\gamma z)^s.
\end{equation}
For each~$z$ in~$\H$, the function~$E_\infty(z, \cdot)$ extends to a meromorphic function defined on~$\C$ that is holomorphic on a neighborhood of the half-plane ${\{s \in \C \: \Re(s) \ge \frac{1}{2} \}}$, except for a simple pole at ${s = 1}$.
Denote by~$C_0^\infty(\Rp)$ the space of~$C^\infty$ complex valued functions defined on~$\Rp$ that are compactly supported, and by~$L^2(\Rp)$ the space of complex valued functions defined on~$\Rp$ that are square integrable with respect to the Lebesgue measure.
Note that~$C_0^\infty(\Rp)$ is a dense subspace of~$L^2(\Rp)$.
The \emph{Eisenstein transform} is the map ${C_0^\infty(\Rp) \to L^2(\Ell(\C),\mhyp)}$ given by
\begin{equation}
  \label{eq:52}
  h \mapsto \frac{1}{4\pi} \int_0^\infty h(t)E_\infty\left(\cdot, \tfrac{1}{2}+it\right) \dd t.
\end{equation}
It is an isometric embedding, see, \emph{e.g.}, \cite[Corollary to Proposition~7.1]{Iwa02}.
Denote by ${\cE_\infty \: L^2(\Rp) \to L^2(\Ell(\C),\mhyp)}$ the isometric embedding extending it.
Then, the space~$\sE_\infty$ appearing in~\eqref{eq:50} is given by ${\sE_\infty = \cE_\infty(L^2(\Rp))}$.

\begin{proof}[Proof of Theorem~\ref{t:equid_complex}]
  The proof begins computing the norm of~$\bT_\ell$ restricted to~$\sE_\infty$, denoted by~$\|\bT_\ell\|_{\sE_\infty}$.
  Put, for all~$s$ in~$\C$ and~$n$ in~$\N$,
  \begin{equation}
    \label{eq:53}
    \sigma_s(n)
    \=
    \sum_{d\in \N, d\mid n}d^s,
  \end{equation}
  so ${\sigma_0 = d}$ and ${\sigma_1 = \sigma}$.
  For every~$s$ in~$\C$,
  \begin{equation}
    \label{eq:54}
    T_\ell E_\infty(\cdot,s)
    =
    \ell^s\sigma_{1-2s}(\ell) E_\infty(\cdot,s),
  \end{equation}
  see, \emph{e.g.}, \cite[p.~117]{Iwa02}.
  Denoting by~${M_\ell \: L^2(\Rp) \to L^2(\Rp)}$ the operator defined by
  \begin{equation}
    \label{eq:55}
    M_\ell(h)(t)
    \=
    \frac{\ell^{\frac{1}{2}+it} \sigma_{-2it}(\ell)}{\sigma(\ell)} h(t),
  \end{equation}
  it follows that ${\bT_\ell \circ \cE_\infty = \cE_\infty \circ M_\ell}$ on~$L^2(\Rp)$ and, since~$\cE_\infty$ is an isometric embedding, ${\|\bT_\ell\|_{\sE_\infty} = \|M_\ell\|_{L^2(\Rp)}}$.
  Since for every~$t$ in~$\Rp$ the inequality ${|\sigma_{-2it}(\ell)| \le d(\ell)}$ holds, ${\|M_\ell\|_{L^2(\Rp)} \le \varrho(\ell)}$.
  To prove the reverse inequality, denote, for each~$\varepsilon$ in~$\Rp$, by~$\bfone_{]0, \varepsilon[}$ the indicator function of~$]0, \varepsilon[$.
  Then, ${\| \varepsilon^{-\frac{1}{2}} \bfone_{]0, \varepsilon[} \|_{L^2(\Rp)} = 1}$ and
  \begin{equation}
    \label{eq:56}
    \|M_\ell\|_{L^2(\Rp)}
    \ge
    \lim_{\varepsilon \to 0^+} \| M_\ell(\varepsilon^{-\frac{1}{2}} \bfone_{]0, \varepsilon[}) \|_{L^2(\Rp)}
    =
    \frac{\ell^{\frac{1}{2}} \sigma_0(\ell)}{\sigma(\ell)}
    =
    \varrho(\ell).
  \end{equation}
  This completes the proof of
  \begin{equation}
    \label{eq:op_norm_conti_spectrum}
    \|\bT_\ell\|_{\sE_\infty}
    =
    \|M_\ell\|_{\sE_\infty}
    =
    \varrho(\ell).
  \end{equation}

  Note that~\eqref{eq:op_norm_conti_spectrum} implies the first inequality in~\eqref{eq:7}.
  To complete the proof of~\eqref{eq:7}, choose the maximal orthogonal system $\left( \varphi_k \right)_{k=0}^\infty$ of eigenfunctions appearing in~\eqref{eq:50} of~$\Delta$ so that~$\varphi_0$ is the constant function equal to~$1$ and such that, for each~$k$ in~$\Nz$, the function~$\varphi_k$ is also an eigenfunction of~$\bT_\ell$, see, \emph{e.g.}, \cite[Section~8.5]{Iwa02}.
  Then, the eigenvalue~$\alpha_k(\ell)$ of~$\varphi_k$ for~$\bT_\ell$ satisfies ${|\alpha_k(\ell)| \le \varrho(\ell) \ell^{\frac{7}{64}}}$, see \cite[first part of Proposition~2 in Appendix~2]{Kim03}.
  Together with~\eqref{eq:op_norm_conti_spectrum}, this yields the second inequality in~\eqref{eq:7}.
  In particular, $\|\bT_\ell\|_0\to 0$ as ${\ell \to \infty}$.
  Combined with Proposition~\ref{p:criterion_convergence_iterates}$(ii)$, this implies the third inequality in~\eqref{eq:7} and completes the proof of~\eqref{eq:7}.

  To prove the last assertion of the theorem, let~$f$ be in~$D(\Ell(\C))$ and~$K$ a compact subset of~$\Ell(\C)$.
  Let~$(A_k)_{k = 0}^\infty$ be the sequence in~$\C$ and~$h$ in~$L^2(\Rp)$ such that
  \begin{equation}
    \label{eq:spectral_decom_f}
    f
    =
    \sum_{k= 0}^\infty A_k \varphi_k + \cE_\infty(h).
  \end{equation}
  The latter series converges absolutely on~$K$, see, \emph{e.g.}, \cite[Theorems~4.7 and~7.3]{Iwa02}.
  In particular,
  \begin{equation}
    \label{eq:57}
    \sup_{z \in K} \sum_{k = 1}^\infty |A_k \varphi_k(z)| + |\cE_\infty(h)(z)|
  \end{equation}
  is finite.
  Denote this constant by~$C_{f, K}$.
  Since~$\varphi_0$ is constant equal to~$1$, the equalities ${A_0 = \int f \dd \mhyp}$ and ${\alpha_0(\ell) = 1}$ hold.
  Using ${\varrho(\ell) \le \on{\bT_\ell}}$ from~\eqref{eq:7} and, for every~$k$ in~$\N$, the elementary inequality ${\alpha_k(\ell) \le \on{\bT_\ell}}$, yields, for all~$n$ in~$\N$ and~$z$ in~$K$,
  \begin{equation}
    \label{eq:58}
    \begin{split}
      \left|\bT_\ell^n f(z)-\int f \dd \mhyp\right|
      & =
        \left| \sum_{k = 1}^\infty A_k \alpha_k(\ell)^n\varphi_k(z) + \cE_\infty(M_\ell^n(h))(z) \right|
      \\ & \le
           \sum_{k = 1}^\infty |\alpha_k(\ell)|^n |A_k \varphi_k(z)| + \sup_{t \in \Rp} \left| \frac{\ell^{\frac{1}{2}+it} \sigma_{-2it}(\ell)}{\sigma(\ell)} \right|^n |\cE_\infty(h)(z)|.
      \\ & \le
           C_{f, K} \on{\bT_\ell}^n.
           \qedhere
    \end{split}
  \end{equation}
\end{proof}

\begin{proof}[Proof of Corollary~\ref{c:equid_complex_bis}]
  Since~$\bT_\ell$ fixes each constant function, the Cauchy--Schwarz inequality implies, for every~$n$ in~$\N$,
  \begin{multline}
    \label{eq:59}
    |\sC_n(f, g)|
    =
    \left| \int \left( \bT_\ell^n f - \int f \dd \mhyp \right) g \dd \mhyp \right|
    \le
    \left\| \bT_\ell^n \left( f - \int f \dd \mhyp \right) \right\|_2 \| g \|_2
    \\ \le
    \on{\bT_\ell^n} \left\| f- \int f \dd \mhyp \right\|_2 \| g \|_2
    \le
    \on{\bT_\ell}^n \norm{f}_2 \| g \|_2.
  \end{multline}
  Since ${\on{\bT_\ell} < 1}$ by~\eqref{eq:7} in Theorem~\ref{t:equid_complex}, it follows that~$(|\sC_n(f, g)|)_{n = 1}^\infty$ decays exponentially.
\end{proof}

\section{$p$-Adic convergence towards the Gauss point}
\label{s:Berkovich}

This section proves Theorem~\ref{t:Equidistribucion}$(i)$ in Section~\ref{ss:Equidistribucion} in its uniform form given by Theorem~\ref{t:canonical_rate} below and shows, in Proposition~\ref{p:canonical_sharpness}, that the exponential convergence rate is sharp.
The proof of Theorem~\ref{t:canonical_rate} relies on results of \cite{HerMenRivI}.
The proof of Proposition~\ref{p:canonical_sharpness} relies on Proposition~\ref{p:Hecke-rate}, which in turn relies on a large deviation estimate.

Throughout this section, fix a prime number~$p$.
Identify the residue field of~$\Cp$ with an algebraic closure~$\Fpalg$ of the field with~$p$ elements~$\Fp$.
Recall that the endomorphism ring of an elliptic curve over~$\Fpalg$ is isomorphic to an order in either a quadratic imaginary extension of~$\Q$ or a quaternion algebra over~$\Q$.
In the former case, the corresponding elliptic curve class is \emph{ordinary}, whereas in the latter it is \emph{supersingular}.

Denote by~$\cO_p$ the ring of integers of~$\Cp$ and ${\pi \: \cO_p \to \Fpalg}$ the reduction map.
An elliptic curve class~$E$ has \emph{good reduction} if there is a representative elliptic curve defined over~$\cO_p$ whose reduction is smooth.
In this case, the reduction is an elliptic curve defined over~$\Fpalg$ whose class~$\tE$ only depends on~$E$ and is the \emph{reduction of~$E$}.
Moreover, $E$ has \emph{ordinary} (resp. \emph{supersingular}) \emph{reduction} if~$\tE$ is ordinary (resp. supersingular).
An elliptic curve class~$E$ over~$\Cp$ has good reduction precisely when its $j$-invariant~$j(E)$ belongs to~$\cO_p$, and otherwise~$E$ has \emph{bad reduction}.
The moduli space~$\Ell(\Cp)$ is thus partitioned into three pairwise disjoint sets: The \emph{bad}, \emph{ordinary}, and \emph{supersingular reduction loci}, denoted by~$\Bad$, $\Ord$, and~$\Sups$, respectively, as in the introduction.
Using ${j \: \Ell(\Cp) \to \Cp}$ to identify~$\Ell(\Cp)$ and~$\Cp$, yields the partition
\begin{displaymath}
  \cO_p
  =
  \Ord \sqcup \Sups.
\end{displaymath}
Denoting by~$\tSups$ the finite subset of~$\Ell(\Fpalg)$ of supersingular classes, the equality ${\Sups = \pi^{-1}(\tSups)}$ holds.

Thus, ${\Sups}$ is a finite union of residue discs of~$\cO_p$ and~$\Ord$ a union of infinitely many residue discs of~$\cO_p$.

Denote by~$\AKber$ the Berkovich affine line over~$\Cp$.
As a set, $\AKber$ is the collection of multiplicative seminorms on the polynomial ring~$\Cp[X]$ that take values in~$\Rzp$ and extend the $p$\nobreakdash-adic norm~$| \cdot |_p$ on~$\Cp$.
The \emph{canonical} or \emph{Gauss point}~$\xcan$ of~$\AKber$, is the Gauss norm
\begin{equation}
  \label{eq:60}
  \sum_{n=0}^N a_nX^n
  \mapsto
  \max \{|a_n|_p \: n \in \{0,\ldots, N\} \}.
\end{equation}
The field~$\Cp $ embeds in~$\AKber$ through the map that maps~$z$ in~$\Cp$ to the seminorm that attaches to a polynomial~$f(X)$ in~$\Cp[X]$ the quantity~$|f(z)|_p$.
See \cite[Section~2.4]{HerMenRivI} for a review of the topology of~$\AKber$, and, \emph{e.g.}, \cite[Chapter~1]{BakRum10} for general properties of~$\AKber$ and~\cite{Ber90} for the general theory of Berkovich spaces.

The following theorem extends Theorem~\ref{t:Equidistribucion}$(i)$.
Recall from Section~\ref{ss:Equidistribucion} that ${\varrho \: \N \to \R}$ is given by ${\varrho(\ell) = \frac{\sqrt{\ell} d(\ell)}{\sigma(\ell)}}$.

\begin{theorem}
  \label{t:canonical_rate}
  Let~$\cU$ be a neighborhood of~$\xcan$ in~$\AKber$ and~$\bfI$ a subset of~$\Cp$ satisfying one of the following conditions.
  \begin{itemize}
  \item
    For some~$\rho$ in~$]1, +\infty[$, the set~$\bfI$ is contained in~$\{ z \in \Cp \: 1 < |z|_p < \rho \}$;
  \item
    The set~$\bfI$ is contained in a residue disc in~$\Ord$;
  \item
    For some~$a$ in~$\Sups$ and~$r$ in~$]0, 1[$, the set~$\bfI$ is contained in ${\{ z \in \Cp \: r < |z - a|_p < 1 \}}$.
  \end{itemize}
  Then, there exists~$C$ in~$\Rp$ such that, for all~$E$ in~$\bfI$ and integers~$\ell$ and~$n$ satisfying ${\ell \ge 2}$ and ${n \ge 1}$,
  \begin{equation}
    \label{eq:con-tasa'}
    \odelta_{T^n_\ell(E)}(\cU)
    \ge
    1 - C
    \begin{cases}
      \varrho(\ell)^n
      & \text{if } \bfI \subseteq \Bad \cup \Ord;
      \\
      \varrho(|\ell|_p^{-1})^n
      & \text{if } \bfI \subseteq \Sups.
    \end{cases}
  \end{equation}
\end{theorem}

The following corollary is a direct consequence of Theorem~\ref{t:canonical_rate} and the definition of~$\xcan$.
Recall the identification of~$\Ell(\Cp)$ with~$\Cp$ via the $j$-invariant map.
In particular, for all~$n$ in~$\N$ and~$E$ in $\Ell(\Cp)$, the probability measure~$\odelta_{T^n_\ell(E)}$ can be seen as a probability measure on~$\Cp$.

\begin{coro}
  \label{c:canonical_convergence}
  Let~$P(X)$ in~$\Cp[X]$ be given by ${P(X) \= \sum_{i=0}^N a_iX^i}$, $\ell$ an integer satisfying ${\ell \ge 2}$, and~$E$ in~$\Ell(\Cp)$.
  If~$E$ belongs to~$\Bad \cup \Ord$ or ${p \mid \ell}$, then
  \begin{equation}
    \label{eq:61}
    \int |P|_p \dd \odelta_{T^n_\ell(E)}
    \to
    \max\{|a_i|_p \: i \in \{0,\ldots,N\} \}
  \end{equation}
  exponentially fast as ${n \to \infty}$.
\end{coro}

The following proposition shows that the exponential convergence rate in Theorems~\ref{t:Equidistribucion}$(i)$ and~\ref{t:canonical_rate} is sharp.

\begin{proposition}
  \label{p:canonical_sharpness}
  For all~$E$ in~$\Ell(\Cp)$ and integer~$\ell$ satisfying ${\ell \ge 2}$, there exists a neighborhood~$\cU$ of~$\xcan$ in~$\AKber$ such that
  \begin{equation}
    \label{eq:canonical_sharpness}
    \lim_{n\to \infty} \frac{1}{n} \log(\odelta_{T^{n}_\ell(E)}(\Cp\ssetminus \cU))
    =
    \begin{cases}
      \log(\varrho(\ell))
      & \text{if } E\in \Bad \cup \Ord;
      \\
      \log(\varrho(|\ell|_p^{-1}))
      & \text{if } E\in \Sups.
    \end{cases}
  \end{equation}
\end{proposition}

For a prime number~$\ell$ distinct from~$p$ and an elliptic curve class~$E$ in~$\Ell_{\ord}(\Qpalg)$, Goren and Kassaei proved that, as ${n \to \infty}$, the mass of~$T_\ell^n(E)$ escapes from each fixed residue disc in the ordinary reduction locus \cite[Proposition~4.1.1]{GorKas}.
Theorems~\ref{t:Equidistribucion}$(i)$ and~\ref{t:canonical_rate} and Proposition~\ref{p:canonical_sharpness} determine the exact exponential escape rate, without restrictions on~$\ell$.
See also \cite[Proposition~4.3.1]{GorKas} for complementary information on the $p$\nobreakdash-adic closure of~$\bigcup_{n=1}^\infty\supp(T_\ell^n(E))$.

The rest of this section proves Theorem~\ref{t:canonical_rate} and Proposition~\ref{p:canonical_sharpness}.
Section~\ref{ss:paperI} collects quantitative estimates whose proofs rely on results from~\cite{HerMenRivI}.
The proofs of Theorem~\ref{t:canonical_rate} and Proposition~\ref{p:canonical_sharpness} occupy Sections~\ref{ss:proof-equid_precisa_Berkovich} and~\ref{ss:canonical_sharpness}, respectively.

\subsection{Quantitative convergence towards the Gauss point}
\label{ss:paperI}

The goal of this section is to prove the following proposition, whose proof is based on results from \cite{HerMenRivI}.
For all~$a$ in~$\Cp$ and~$r$ in~$\R_{> 0}$, define
\begin{equation}
  \label{eq:62}
  \begin{aligned}
    \bfD(a,r)
    \=
    \{x\in \Cp:|x-a|_p<r\}
    \text{ and }
    \bfD^\infty(a,r)
    \=
    \{x\in \Cp:|x-a|_p>r\}.
  \end{aligned}
\end{equation}
Given a divisor~$\fD$ in~$\Div(\Ell(\K))$, written as ${\fD = \sum_{E\in \Ell(\K)} n_EE}$, and a subset~$A$ of~$\Ell(\Cp)$, define
\begin{equation}
  \label{eq:63}
  \fD|_A
  \=
  \sum_{E \in A} n_E E.
\end{equation}

\begin{proposition}
  \label{p:paperI}
  Let~$R$ be in~$]1, +\infty[$ and~$A$ a finite subset of~$\cO_p$, and put
  \begin{equation}
    \label{eq:64}
    \bfB
    \=
    \bfD(0, R) \cap \bigcap_{a\in A} \bfD^\infty(a, R^{-1}).
  \end{equation}
  Then, for every set~$\bfI$ as in Theorem~\ref{t:canonical_rate} there exists~$C$ in~$[1, +\infty[$ such that, for all~$E$ in~$\bfI$ and~$\ell$ in~$\N$,
  \begin{equation}
    \label{conteo1}
    \deg(T_\ell(E)|_{\Cp \ssetminus \bfB})
    \le
    \begin{cases}
      C\sqrt{\ell} d(\ell)
      & \text{if } \bfI \subseteq \Bad;
      \\
      C d(\ell)
      & \text{if } \bfI \subseteq \Ord;
      \\
      C\sigma(\ell \cdot |\ell|_p)
      & \text{if } \bfI \subseteq \Sups.
    \end{cases}
  \end{equation}
\end{proposition}

In the ordinary case, the proof of Proposition~\ref{p:paperI} relies on the following variant of~\cite[Lemma~5.5]{HerMenRivI}.
For all~$e$ and~$e'$ in~$\Ell(\Fpalg)$, denote by~$\Hom(e,e')$ the additive group of isogenies from~$e$ to~$e'$, and, for every~$n$ in~$\N$, by~$\Hom_{n}(e,e')$ the subset of~$\Hom(e,e')$ of isogenies of degree~$n$.

\begin{lemma}
  \label{l:ordinary}
  For all ordinary~$e$ and~$e'$ in~$\Ell(\Fpalg)$, there exists~$C_{e, e'}$ in~$[1, +\infty[$ such that, for every~$m$ in~$\N$,
  \begin{equation}
    \label{eq:65}
    \# \Hom_m(e, e')
    \le
    C_{e, e'} d(m).
  \end{equation}
\end{lemma}

\begin{proof}
  Since~$\Hom(e,e)$ coincides with~$\End(e)$ and~$e$ is ordinary, $\Hom(e, e)$ is a ring isomorphic to an order in a quadratic imaginary extension~$K$ of~$\Q$.
  Choose the isomorphism so that the degree of an isogeny coincides with the field norm of the corresponding element of~$K$, see, \emph{e.g.}, \cite[Chapter~V, Theorem~3.1]{Sil09}.
  Denoting by~$\cO_K$ the ring of integers of~$K$, it follows that, for each~$m$ in~$\N$, the number~$\# \Hom_m(e, e)$ is less than or equal to the number of elements of~$\cO_K$ with field norm~$m$.
  Denoting by~$R_K(m)$ the number of ideals in the ring of integers~$\cO_K$ of~$K$ of norm~$m$, this yields
  \begin{equation}
    \label{eq:66}
    \# \Hom_m(e, e)
    \le
    (\#\cO_K^\times) \cdot R_K(m)
    \le
    6R_K(m).
  \end{equation}
  Denote by~$\bfone$ the arithmetic function that is constant equal to~$1$ and by~$\psi_K$ is the quadratic character of~$K$.
  Then, the arithmetic function~$R_K$ equals the Dirichlet convolution~$\psi_K \ast \bfone$, see, \emph{e.g.}, \cite[(2-13)]{HerMenRivI}.
  It follows that, for every~$m$ in~$\N$, the inequality ${R_K(m) \le d(m)}$ holds.
  Together with~\eqref{eq:66}, this implies the desired estimate with ${C_{e, e} = 6}$ when ${e' = e}$.

  Suppose ${e' \neq e}$.
  If the group~$\Hom(e, e')$ is trivial, then~\eqref{eq:65} holds trivially.
  Suppose~$\Hom(e, e')$ is nontrivial, choose a nonzero isogeny~$\phi_0$ in~$\Hom(e',e)$, and put ${m_0 \= \deg(\phi_0)}$.
  Since for each~$m$ in~$\N$ the map ${\Hom_m(e, e') \to \Hom_{m_0 m}(e,e)}$ given by ${\phi\mapsto \phi_0 \circ \phi}$ is injective, \eqref{eq:66} with~$m$ replaced by~$m_0m$ yields
  \begin{equation}
    \label{eq:67}
    \Hom_m(e, e')
    \le
    \Hom_{m_0 m}(e,e)
    \le
    6R_K(m_0m)
    \le
    6d(m_0m)
    \le
    6d(m_0)d(m).
  \end{equation}
  This completes the proof of the lemma with ${C = 6d(m_0)}$.
\end{proof}

\begin{proof}[Proof of Proposition~\ref{p:paperI}]
  Let~$R$ be in~$]1, +\infty[$, $A$ a finite subset of~$\cO_p$, and~$\bfB$ the set defined by~\eqref{eq:64}.

  If ${\bfI \subseteq \Sups}$, then~\eqref{conteo1} is a direct consequence of \cite[Proposition~5.6]{HerMenRivI}.

  Suppose ${\bfI \subseteq \Bad}$ and let~$\rho$ in~$]1, +\infty[$ be such that ${\bfI \subseteq \{ z \in \Cp \: 1 < |z|_p < \rho \}}$.
  By \cite[Proof of Proposition~5.1]{HerMenRivI},
  \begin{equation}
    \label{eq:68}
    \deg(T_\ell(E)|_{\Cp \ssetminus \bfB})
    =
    \deg(T_\ell(E)|_{\bfD^\infty(0, R)})
    \le
    \sqrt{\frac{\log |j(E)|_p}{\log R}} \sqrt{\ell}d(\ell)
    \le
    \sqrt{\frac{\log \rho}{\log R}} \sqrt{\ell}d(\ell).
  \end{equation}

  It remains to consider the case where~$\bfI$ is contained in a residue disc in~$\Ord$.
  Denote by~$e$ the common reduction in~$\Ell(\Fpalg)$ of each element of~$\bfI$, and, for each~$a$ in~$A$, by~$e_a$ the common reduction in~$\Ell(\Fpalg)$ of each element of~$\bfD(a, 1)$.
  Moreover, put ${e_a^{(0)} \= e_a}$ and, for each~$i$ in~$\N$, denote by~$e_a^{(i)}$ the elliptic curve class in~$\Ell(\Fpalg)$ obtained from~$e_a$ after~$i$ iterations of the Frobenius morphism.
  For each~$i$ in~$\Nz$, put ${\bfD_{a, i} \= \pi^{-1}(e_a^{(i)})}$ and let~$C_{a, i}$ be given by Lemma~\ref{l:ordinary} with ${e' = e_a^{(i)}}$.
  Let~$t_a$ be the smallest integer in~$\N$ satisfying ${e^{(t_a)}_a = e_a}$, and put
  \begin{equation}
    \label{eq:69}
    \bfO_a
    \=
    \bfD_{a, 0} \cup \cdots \cup \bfD_{a, t_a - 1}
    \text{ and }
    C_a
    \=
    \sum_{i = 0}^{t_a - 1} C_{a, i}.
  \end{equation}
  Let~$k$ in~$\Nz$ and~$n$ in~$\N$ be such that ${\ell = p^kn}$ and ${p \nmid n}$.
  By \cite[Proposition~5.3]{HerMenRivI}, there is~$C_a'$ in~$\Rp$ such that, for every~$E'$ in~$\bfO_a$,
  \begin{equation}
    \label{inter}
    \supp(T_{p^n}(E'))
    \subseteq
    \bfO_a
    \text{ and }
    \deg (T_{p^n}(E')|_{\bfD(a, \rho^{-1})})
    \le
    C_a' d(p^n).
  \end{equation}
  Combining \cite[(5-8)]{HerMenRivI} and Lemma~\ref{l:ordinary} yields, for every~$E$ in~$\bfI$,
  \begin{equation}
    \label{eq:70}
    \deg(T_{n}(E)|_{\bfO_a})
    =
    \sum_{i = 0}^{t_a - 1} \deg(T_{n}(E)|_{\bfD_{a, i}})
    \le
    \sum_{i = 0}^{t_a - 1} \# \Hom_{n}(e,e_a^{(i)})
    \le
    C_a d(n).
  \end{equation}
  Combined with~\eqref{eq:coprime} and~\eqref{inter}, this implies
  \begin{align*}
    \deg(T_{p^kn}(E)|_{\bfD(a, \rho^{-1})})
    \le C_a' d(p^k) \sup_{E'\in \bfO_a} \deg (T_{n}(E')|_{\bfO_a})
    \le C_a' C_a d(\ell).
  \end{align*}
  Summing over~$a$ in~$A$ yields the desired estimate with ${C = \sum_{a \in A} C_a' C_a}$.
\end{proof}

\subsection{Proof of Theorem~\ref{t:canonical_rate}}
\label{ss:proof-equid_precisa_Berkovich}
Let~$\cU$ and~$\bfI$ be as in the statement of the theorem, and~$R$ in~$]1, +\infty[$ and a finite subset~$A$ of~$\Op$ be such that the set~$\bfB$ defined by~\eqref{eq:64} is contained in~$\cU$, see, \emph{e.g.}, \cite[Section~2D]{HerMenRivI}.
Let~$C$ be given by Proposition~\ref{p:paperI} and let ${q_1^{m_1} \cdots q_r^{m_r}}$ be the prime factorization of~$\ell$, so~$r$ and $m_1$, \ldots, $m_r$ belong to~$\N$, $q_1$, \ldots, $q_r$ are pairwise distinct prime numbers, and ${\ell = q_1^{m_1} \cdots q_r^{m_r}}$.
Using~\eqref{eq:coprime} as well as \eqref{id} and~\eqref{suma total} in Proposition~\ref{p:Hecke-composition}, yields
\begin{multline}
  \label{eq:factorizacion}
  \sigma(\ell)^n \odelta_{T^n_\ell(E)}(\AKber \ssetminus \cU)
  \le
  \deg(T^n_\ell(E)|_{\Cp \ssetminus \bfB})
  \\ =
  \sum_{j_1=0}^{\left\lfloor \frac{m_1n}{2} \right\rfloor} \cdots \sum_{j_r=0}^{\left\lfloor \frac{m_rn}{2} \right\rfloor}a_{m_1,n}(j_1) \cdots a_{m_r,n}(j_r) q_1^{j_1} \cdots q_r^{j_r} \deg (T_{q_1^{m_1n-2j_1} \cdots q_r^{m_rn-2j_r}}(E)|_{\Cp \ssetminus \bfB}).
\end{multline}
If ${\bfI \subseteq \Ord \cup \Bad}$, then Proposition~\ref{p:Hecke-composition}$(ii)$, \eqref{conteo1} in Proposition~\ref{p:paperI}, and the multiplicativity of~$d$, yield
\begin{equation}
  \label{eq:71}
  \begin{split}
    \sigma(\ell)^n \odelta_{T^n_\ell(E)}(\AKber \ssetminus \cU)
    & \le
      C \prod_{k=1}^r \left(\sum_{j_k=0}^{\left\lfloor \frac{m_kn}{2} \right\rfloor} a_{m_k,n}(j_k)q_k^{j_k} \sqrt{q_k^{m_kn-2j_k}}d(q_k^{m_kn-2j_k}) \right)
    \\ & =
         C \prod_{k=1}^r \left(\sqrt{q_k^{m_kn}} \sum_{j_k=0}^{\left\lfloor \frac{m_kn}{2} \right\rfloor} a_{m_k,n}(j_k)(m_kn-2j_k+1) \right)
    \\ & =
         C(\sqrt{\ell})^n \prod_{k=1}^r (m_k+1)^n
    \\ & =
         C(\sqrt{\ell} d(\ell))^n.
  \end{split}
\end{equation}
It remains to consider the case were ${\bfI \subseteq \Sups}$.
If~$\ell$ is not divisible by~$p$, then ${\varrho(|\ell|_p^{-1}) = 1}$, and the desired estimate holds trivially.
Suppose~$\ell$ is divisible by~$p$.
Re-enumerating if necessary, suppose ${q_1 = p}$.
Using~\eqref{conteo1} in Proposition~\ref{p:paperI}, Proposition~\ref{p:Hecke-composition}, and the multiplicativity of~$\sigma$, yields
\begin{equation}
  \label{eq:72}
  \begin{split}
    \sigma(\ell)^n \odelta_{T^n_\ell(E)}(\AKber \ssetminus \cU)
    & \le
      C\left(\sum_{j=0}^{\left\lfloor \frac{m_1 n}{2} \right\rfloor}a_{m_1,n}(j)p^{j} \right) \prod_{k=2}^r\left(\sum_{j_k=0}^{\left\lfloor \frac{m_kn}{2} \right\rfloor}a_{m_k,n}(j_k) q_k^{j_k} \sigma(q_k^{m_kn-2j_k}) \right)
    \\ & \le
         Cp^{\frac{m_1 n}{2}}(m_1 + 1)^n \prod_{k=2}^r\left(\sum_{j_k=0}^{\left\lfloor \frac{m_kn}{2} \right\rfloor}a_{m_k,n}(j_k) q_k^{j_k} \sigma(q_k^{m_kn-2j_k}) \right)
    \\ & =
         Cp^{\frac{m_1 n}{2}}(m_1 + 1)^n \prod_{k=2}^r \sigma(q_k^{m_k})^n
    \\ & =
         C\left(p^{\frac{m_1}{2}} d(p^{m_1}) \sigma(\ell / p^{m_1}) \right)^n.
         \qedhere
  \end{split}
\end{equation}

\subsection{Sharpness}
\label{ss:canonical_sharpness}

This section proves Proposition~\ref{p:canonical_sharpness}.
The proof relies on the following consequence of Proposition~\ref{p:Hecke-rate}.

\begin{lemma}
  \label{l:Hecke-recurrence}
  For all~$E$ in~$\Ell(\K)$ and integer~$\ell$ satisfying ${\ell \ge 2}$,
  \begin{align}
    \label{eq:Hecke-recurrence}
    \liminf_{n \to \infty} \frac{1}{2n} \log (\odelta_{T^{2n}_\ell(E)}(\{ E \}))
    & \ge
      \log \varrho(\ell)
      \intertext{and}
      \label{eq:Hecke-recurrence_bis}
      \liminf_{n \to \infty} \frac{1}{n} \log (\odelta_{T^{n}_\ell(E)}(\supp(T_\ell(E))))
    & \ge
      \log \varrho(\ell).
  \end{align}
\end{lemma}

\begin{proof}
  The statements~\eqref{eq:Hecke-recurrence} and~\eqref{eq:Hecke-recurrence_bis} are equivalent to
  \begin{align}
    \label{eq:cor_a}
    \liminf_{n\to \infty} \frac{1}{2n} \log \deg\left(T_\ell^{2n}(E)|_{\{E\}} \right)
    & \ge
      \log (\sqrt{\ell}d(\ell))
      \intertext{and}
      \label{eq:cor_a_bis}
      \liminf_{n\to \infty} \frac{1}{n} \log \deg\left(T_\ell^{n}(E)|_{\supp(T_\ell(E))} \right)
    & \ge
      \log (\sqrt{\ell}d(\ell)),
  \end{align}
  respectively.

  Let ${q_1^{m_1} \cdots q_r^{m_r}}$ be the prime factorization of~$\ell$, so~$r$ and $m_1$, \ldots, $m_r$ belong to~$\N$, $q_1$, \ldots, $q_r$ are pairwise distinct prime numbers, and ${\ell = q_1^{m_1} \cdots q_r^{m_r}}$.
  Using~\eqref{eq:coprime} as well as~\eqref{id} in Proposition~\ref{p:Hecke-composition}, yields, for every~$n$ in~$\N$,
  \begin{equation}
    \label{eq:cor_b}
    T^{2n}_\ell(E)
    \\ =
    \sum_{j_1=0}^{ m_1n } \cdots \sum_{j_r=0}^{ m_rn }a_{m_1,2n}(j_1) \cdots a_{m_r,2n}(j_r) q_1^{j_1} \cdots q_r^{j_r} T_{q_1^{2m_1n-2j_1} \cdots q_r^{2m_rn-2j_r}}(E).
  \end{equation}
  Recall that, for every~$N$ in~$\N$, the multiplication-by-$N$ map defines an isogeny~$E\to E$ of degree $N^2$, so
  \begin{equation}
    \label{eq:73}
    \deg(T_{N^2}(E)|_{\{E\}})
    \ge
    1.
  \end{equation}
  Combined with~\eqref{eq:cor_b}, this yields
  \begin{equation}
    \deg(T^{2n}_\ell(E)|_{\{E\}})
    \ge
    \sum_{j_1=0}^{ m_1n } \cdots \sum_{j_r=0}^{ m_rn }a_{m_1,2n}(j_1) \cdots a_{m_r,2n}(j_r) q_1^{j_1} \cdots q_r^{j_r}
    =
    \prod_{i=1}^r \sum_{j=1}^{m_in}a_{m_i,2n}(j)q_i^j.
  \end{equation}
  By Proposition~\ref{p:Hecke-rate},
  \begin{equation}
    \label{eq:74}
    \begin{split}
      \liminf_{n\to \infty} \frac{1}{2n} \log \deg\left(T_\ell^{2n}(E)|_{\supp(T_\ell(E))} \right)
      & \ge
        \liminf_{n\to \infty} \frac{1}{2n} \log(\deg(T^{2n}_\ell(E)|_{\{E\}}))
      \\ & \ge
           \sum_{i=1}^r\lim_{n\to \infty} \frac{1}{2n} \log\left( \sum_{j=1}^{m_in}a_{m_i,2n}(j)q^j \right)
      \\ & =
           \sum_{i=1}^r \log (q^{\frac{m_i}{2}} (m_i+1))
      \\ & =
           \log ( \sqrt{\ell} d(\ell)).
    \end{split}
  \end{equation}
  This proves~\eqref{eq:cor_a} and hence~\eqref{eq:Hecke-recurrence}.

  To prove~\eqref{eq:Hecke-recurrence_bis}, note that, for every~$E'$ in~$\supp(T_\ell(E))$, the elliptic curve class~$E$ belongs to~$T_\ell(E')$.
  Thus, \eqref{eq:Hecke-recurrence} with~$E$ replaced by~$E'$ yields
  \begin{equation}
    \label{eq:75}
    \begin{split}
      \liminf_{n\to \infty} \frac{1}{2n + 1} \log \deg\left(T_\ell^{2n + 1}(E)|_{\supp(T_\ell(E))} \right)
      & \ge
        \liminf_{n\to \infty} \frac{1}{2n + 1} \log(\deg(T^{2n + 1}_\ell(E)|_{\{E'\}}))
      \\ & \ge
           \liminf_{n\to \infty} \frac{1}{2n + 1} \log(\deg(T^{2n}_\ell(E')|_{\{E'\}}))
      \\ & \ge
           \log ( \sqrt{\ell} d(\ell)).
    \end{split}
  \end{equation}
  Together with~\eqref{eq:74} this implies~\eqref{eq:cor_a_bis} and hence~\eqref{eq:Hecke-recurrence_bis}.
\end{proof}

\begin{proof}[Proof of Proposition~\ref{p:canonical_sharpness}]
  For every neighborhood~$\cU$ of~$\xcan$ in~$\AKber$, the upper bound
  \begin{equation}
    \label{eq:76}
    \limsup_{n\to \infty} \frac{1}{n} \log(\odelta_{T^{n}_\ell(E)}(\AKber \ssetminus \cU))
    \le
    \begin{cases}
      \log(\varrho(\ell))
      & \text{if } E\in \Bad \cup \Ord;
      \\
      \log(\varrho(|\ell|_p^{-1}))
      & \text{if } E\in \Sups
    \end{cases}
  \end{equation}
  follows directly from Theorem~\ref{t:canonical_rate}.

  If~$E$ belongs to ${\Bad \cup \Ord}$ and~$\cU$ is a neighborhood of~$\xcan$ in~$\AKber$ disjoint from~$\supp(T_\ell(E))$, then~\eqref{eq:Hecke-recurrence_bis} in Lemma~\ref{l:Hecke-recurrence} yields
  \begin{equation}
    \label{eq:77}
    \liminf_{n\to \infty} \frac{1}{n} \log(\odelta_{T^{n}_\ell(E)}(\AKber \ssetminus \cU))
    \ge
    \liminf_{n\to \infty} \frac{1}{n} \log(\odelta_{T^{n}_\ell(E)}(\supp(T_\ell(E))))
    \ge
    \log(\varrho(\ell)).
  \end{equation}
  Together with~\eqref{eq:76} this proves the proposition when~$E$ belongs to ${\Bad \cup \Ord}$.

  Suppose~$E$ belongs to~$\Sups$.
  By~\cite[Proposition~5.6$(ii)$]{HerMenRivI}, for every~$m$ in~$\Nz$, there is a neighborhood~$\cU_m$ of~$\xcan$ in~$\AKber$ such that, for every~$k$ in~$\N$ not divisible by~$p$, the inclusion ${\supp(T_{p^m k}(E)) \subseteq \Cp\ssetminus \cU_m}$ holds.
  If~$\ell$ is not divisible by~$p$, then ${\varrho(|\ell|_p^{-1}) = 1}$ and, for every~$n$ in~$\N$,
  \begin{equation}
    \label{eq:78}
    \log(\odelta_{T^{n}_\ell(E)}(\AKber \ssetminus \cU_0))
    =
    0.
  \end{equation}
  This implies the proposition when~$\ell$ is not divisible by~$p$.
  Suppose~$\ell$ is divisible by~$p$, and let~$m$ and~$k$ in~$\N$ be such that ${\ell = p^m k}$.
  Then,
  \begin{multline}
    \label{eq:79}
    \sigma(\ell)^n \odelta_{T^n_\ell(E)}(\AKber \ssetminus \cU_m)
    =
    \deg(T_k^n(T_{p^m}^n(E))|_{\AKber \ssetminus \cU_m})
    \\ \ge
    \sigma(k)^n \deg(T_{p^m}^n(E)|_{\supp(T_{p^m}(E))})
    =
    \sigma(\ell)^n \odelta_{T_{p^m}^n(E)}(\supp(T_{p^m}(E))).
  \end{multline}
  Combined with~\eqref{eq:Hecke-recurrence_bis} in Lemma~\ref{l:Hecke-recurrence}, this implies
  \begin{equation}
    \label{eq:80}
    \liminf_{n\to \infty} \frac{1}{n} \log(\odelta_{T^{n}_\ell(E)}(\AKber \ssetminus \cU_m))
    \ge
    \log \varrho(p^m)
    =
    \log \varrho(|\ell|_p^{-1}).
  \end{equation}
  Together with~\eqref{eq:76} this proves the proposition when~$E$ belongs to~$\Sups$.
\end{proof}

\section{$p$-Adic limits on the supersingular locus}
\label{ss:p-adic-case}

This section proves Theorem~\ref{t:Equidistribucion}$(ii)$ in Section~\ref{ss:Equidistribucion}.
A key step is to establish the spectral gap property for the action of~$\bT_\ell$ on a suitable space of locally constant functions (Theorem~\ref{t:hecke-dynamics_sups}$(i)$).

Throughout this section, fix a prime number~$p$.
Recall that~$\tSups$ is the finite set of isomorphism classes of supersingular elliptic curves over $\overline{\F}_p$.
For each~$\ss$ in~$\tSups$, denote by~$\Dss$ the set of those~$E$ in~$\Ell(\Cp)$ having good reduction, and such that the reduced elliptic curve belongs to the class~$\ss$.
Thus,
\begin{equation}
  \label{eq:81}
  \Sups
  =
  \bigsqcup_{\ss \in \tSups} \bfD_{\ss}.
\end{equation}

Section~\ref{ss:deformation-spaces} describes, for each~$\ss$ in~$\tSups$, a finite degree covering ${\Piss \: \hDss \rightarrow \Dss}$, together with an action of the group~$(\End(\ss) \otimes \Zp)^\times$ on~$\hDss$ that relies on the work of Gross and Hopkins on deformation spaces of formal modules in~\cite{HopGro94b}.
Then, Section~\ref{ss:homogeneo} recalls from~\cite{HerMenRivII} the construction of the homogeneous measures on each~$\hDss$ that, after push forward under~$\Piss$ and averaging over~$\ss$ in~$\tSups$, give rise to the measures on partial Hecke orbits appearing in Theorem~\ref{t:Equidistribucion}$(ii)$ in Section~\ref{ss:Equidistribucion}.
Section~\ref{s:equid_in_sups_locus} states and proves the spectral gap property (Theorem~\ref{t:hecke-dynamics_sups}$(i)$) and deduces Theorem~\ref{t:Equidistribucion}$(ii)$.

\subsection{Deformation spaces}
\label{ss:deformation-spaces}
Refer to \cite{Fro68} and \cite{Haz78} for the general theory of formal groups and modules.
Let~$\Q_{p^2}$ denote the unique unramified, quadratic extension of~$\Qp$ inside~$\Cp$, and~$\Z_{p^2}$ its ring of integers.
Moreover, identify the residue field of~$\Q_{p^2}$ with the unique quadratic extension~$\F_{p^2}$ of~$\Fp$ inside~$\Fpalg$.

Let $\pi_0 \: \Zp \rightarrow R_0$ be a $\Zp$-algebra.
For all formal groups~$\cF$ and~$\cF'$ defined over~$R_0$, denote by $\Hom_{R_0}(\cF,\cF')$ the set of morphisms $\cF\to \cF'$ defined over~$R_0$ and by~$\Iso_{R_0}(\cF,\cF')$ the set of isomorphisms ${\cF\to \cF'}$ defined over~$R_0$.
Moreover, put ${\End_{R_0}(\cF) \= \Hom_{R_0}(\cF,\cF)}$ and ${\Aut_{R_0}(\cF) \= \Iso_{R_0}(\cF,\cF)}$.

A \emph{formal $\Zp$\nobreakdash-module over~$R_0$} is a formal group~$\cF$ over~$R_0$ of dimension~1, together with a ring homomorphism~$\theta \: \Zp \to \End_{R_0}(\cF)$ such that, in coordinates, for every~$\ell$ in~$\Zp$, the congruence ${\theta(\ell)(X) \equiv \piem(\ell) X \mod X^2}$ holds.

For every~$\ss$ in~$\tSups$, let~$\ss_0$ be a representative elliptic curve defined over~$\F_{p^2}$ as in~\cite[Lemma~2.5]{HerMenRivII}.
Denote by~$\cF_\ss$ the formal group defined over~$\F_{p^2}$ associated with~$\ss_0$, endowed with its natural structure of formal $\Zp$\nobreakdash-module, see, \emph{e.g.}, \cite{Blu98}.
Let~$R_0$ be the ring of integers of a finite extension of $\Q_{p^2}$ and denote by~$\cM_0$ its maximal ideal and~$\rfk$ its residue field.
A \emph{deformation of~$\Fss$ over~$R_0$} is a pair~$(\cF,\alpha)$, where~$\cF$ is a formal $\Zp$\nobreakdash-module over~$R_0$ and ${\alpha \: \tcF \to \Fss}$ is an isomorphism of formal~$\Zp$\nobreakdash-modules defined over~$\rfk$.
Here, $\tcF$ is the formal group over~$\rfk$ obtained as the base change of~$\cF$ under the reduction map ${R_0 \to \rfk}$.
Two such deformations~$(\cF,\alpha)$ and~$(\cF',\alpha')$ are \emph{isomorphic}, if there exists an isomorphism~$\varphi$ in~${\Iso_{R_0} ( \cF , \cF')}$ with reduction~$\tvarphi$ such that ${\alpha'\circ \tvarphi = \alpha}$.

Denote by~$\Xss(R_0)$ the set of isomorphism classes of deformations of~$\Fss$ over~$R_0$.
Then, a consequence of the work of Gross and Hopkins is that there exists a bijection
\begin{equation}
  \label{eq:parametrization-X-e}
  \cM_{0} \rightarrow \Xss(R_0)
\end{equation}
that is functorial in~$R_0$, see \cite[Section~12]{HopGro94b} and \cite[Section~2.5]{HerMenRivII} for details.
Use this bijection to endow~$\Xss(R_0)$ with the distance coming from the distance on~$\cM_{0}$ inherited from~$R_0$.

Given~$\ss$ and~$\sspr$ in~$\tSups$, put
\begin{equation}
  \label{eq:82}
  \bfG_{\ss, \ss'}
  \=
  \Iso_{\F_{p^2}}(\Fss,\Fsspr).
\end{equation}
The group~$\Gss$, defined in Section~\ref{ss:random_walks}, equals~$\bfG_{\ss, \ss}$.
Given~$R_0$ as above, consider the map
\begin{equation}
  \label{eq:action_Iso}
  \begin{array}{rcl}
    \Iso_{\rfk}(\Fss,\Fsspr) \times \Xss(R_0)& \to & \Xsspr(R_0)
    \\
    (\beta, (\cF,\alpha)) & \mapsto & \beta\cdot(\cF,\alpha) \= (\cF,\beta \circ \alpha).
  \end{array}
\end{equation}
It induces an action of~$\Aut_{\rfk}(\Fss)$ on~$\Xss(R_0)$.

Consider
\begin{displaymath}
  \sK
  \=
  \{ \text{finite extensions of~$\Q_{p^2}$ inside~$\Cp$} \}
\end{displaymath}
as a directed set with respect to inclusion, and, for each~$\cK$ in~$\sK$, the parametrization~\eqref{eq:parametrization-X-e} with ${R_0 = \OK}$.
Taking a direct limit over~$\sK$ and then a completion yields a set~$\hDss$ parametrized by the maximal ideal~$\cM_p$ of the ring of integers~$\Op$ of $\Cp$, a map ${\bfG_{\ss, \ss'} \times \hDss \to \hDsspr}$ extending~\eqref{eq:action_Iso}, and an action of~$\Gss$ on~$\hDss$.
By~\cite[Proposition~14.13]{HopGro94b}, the action of the subgroup~$\Zp^{\times}$ of~$\Gss$ on~$\hDss$ is trivial.
Additionally, the following result follows from~\cite[Section~14]{HopGro94b}.

\begin{lemma}
  \label{l:action_isometries}
  For all~$\ss$ and~$\sspr$ in~$\tSups$ and~$\beta$ in~$\bfG_{\ss, \ss'}$, the map $\hDss \to \hDsspr$ given by~$x\mapsto \beta\cdot x$ as in~\eqref{eq:action_Iso}, is given by a rigid analytic isomorphism defined over~$\Z_{p^2}$ and, in particular, it is an isometry.
\end{lemma}

Let~$\cK$ be in~$\sK$ and let~$x$ in~$\Xss(\OK)$ represent a deformation $(\cF,\alpha)$.
By the so-called Woods-Hole Theory, see \cite[Section~6]{LubSerTat64} or \cite[Theorem~4.1]{ColMcM10}, there is an elliptic curve~$E$ defined over~$\OK$ and an isomorphism ${\alpha \: \tE \rightarrow \ss}$ defined over the residue field of~$\cK$, inducing an isomorphism ${\halpha \: \widetilde{\cF}_E \rightarrow \Fss}$ such that $(\cF,\alpha)$ and $(\FE, \halpha)$ are isomorphic deformations.
The assignment $\Pi_{e,\cK}(x) \= E$ determines a map
\begin{equation}
  \label{eq:83}
  \Pi_{e,\cK} \: \Xss(\OK) \rightarrow \Dss
\end{equation}
that, after taking direct limits and completions and letting ${\delta_\ss \= \# \Aut(\ss)/2}$, induces a rigid analytic ramified covering ${\Piss \: \hDss\rightarrow \Dss}$ of degree~$\delta_\ss$.
Moreover, for all~$x$ in~$\hDss$ and~$E$ in~$\Dss$,
\begin{equation}
  \label{eq:84}
  \min \{ |x - x'|_p \: x' \in \Piss^{-1}(E) \}^{\delta_{\ss}}
  \le
  |j(\Piss(x)) - j(E)|_p
  \le
  \min \{ |x - x'|_p \: x' \in \Piss^{-1}(E) \},
\end{equation}
see \cite[Theorem~2.7]{HerMenRivII}.

\subsection{Homogeneous measures and partial Hecke orbits}
\label{ss:homogeneo}
Let~$\ss$ and~$\ss'$ be in~$\tSups$.
For each isogeny~$\phi$ in~$\Hom(\ss, \ss')$, denote by~$\overline{\phi}$ its dual isogeny.
Consider the $\Z$-bilinear map
\begin{displaymath}
  \begin{array}{rcl}
    \langle \ , \ \rangle \: \Hom(\ss, \ss') \times \Hom(\ss, \ss') & \to & \End(\ss)
    \\
    (\phi,\phi') & \mapsto & \langle \phi,\phi' \rangle \= \overline{\phi} \phi' + \overline{\phi'}{\phi}.
  \end{array}
\end{displaymath}
It takes values in~$\Z$ and induces the positive definite quadratic form
\begin{displaymath}
  \begin{array}{rcl}
    Q_{\ss, \ss'} \: \Hom(\ss, \ss') & \to & \Z
    \\
    \phi & \mapsto & Q_{\ss, \ss'}(\phi) \= \frac{1}{2} \langle \phi,\phi\rangle.
  \end{array}
\end{displaymath}

Put ${\bfR_{\ss, \ss'} \= \Hom_{\Fpalg}(\Fss,\Fsspr)}$.
The natural map ${\Hom(\ss, \ss') \to \bfR_{\ss, \ss'}}$, denoted by ${\phi \mapsto \hphi}$, extends to an isomorphism of $\Zp$\nobreakdash-modules
\begin{equation}
  \label{eq:85}
  \Hom(\ss, \ss') \otimes \Zp \xrightarrow{\sim} \bfR_{\ss, \ss'}.
\end{equation}
Extend accordingly~$Q_{\ss, \ss'}$ to a quadratic form on~$\bfR_{\ss, \ss'}$ taking values in~$\Zp$.
Then, for all~$\ss''$ in~$\Sups$, $\phi$ in~$\bfR_{\ss,\ss'}$, and~$\psi$ in~$\bfR_{\ss', \ss''}$,
\begin{equation}
  \label{eq:86}
  Q_{\ss, \ss''}(\psi \phi)
  =
  Q_{\ss',\ss''}(\psi)Q_{\ss, \ss'}(\phi).
\end{equation}

Put ${\Rss \= \bfR_{\ss,\ss}}$ and ${\Bss \= \Rss\otimes \Qp}$.
Then, ${\Bss}$ is a division quaternion algebra over~$\Qp$ and~$\Rss$ is naturally identified with its maximal $\Zp$-order.
Denote by ${g \mapsto \overline{g}}$ the involution of~$\Bss$ and by~$\nr$ its reduced norm, defined by ${\nr(g) \= g\overline{g}}$.
The function ${d_{\Bss} \: \Bss \times \Bss \to \Rzp}$ defined by ${d_{\Bss}(g,g') \= |\nr(g - g')|_p^{\frac{1}{2}}}$ is an ultrametric distance on~$\Bss$.
It makes~$\Bss$ into a topological algebra over~$\Qp$.
Endow~$\Rss$ with the induced distance and~$\bfR_{\ss, \ss'}$ with the unique distance such that, for every~$\varphi_0$ in~$\bfG_{\ss, \ss'}$, the map ${\Rss \to \bfR_{\ss, \ss'}}$ defined by ${\psi \mapsto \varphi_0\circ \psi}$ is an isometry.

For each nonzero~$\ell$ in~$\Zp$, the set~$S_\ell(\ss, \ss')$ defined by
\begin{displaymath}
  S_\ell(\ss, \ss')
  \=
  \{\varphi \in \bfR_{\ss, \ss'} \: Q_{\ss, \ss'}(\varphi)=\ell\},
\end{displaymath}
is called a \emph{supersingular sphere}.
It is nonempty and compact, and it is contained in~$\bfG_{\ss, \ss'}$ if and only if~$\ell$ belongs to~$\Zp^{\times}$ \cite[Proposition~6.2$(i)$]{HerMenRivII}.

The following formula describes the action on~$\Sups$ of the Hecke correspondences of index not divisible by~$p$.

\begin{lemma}[\textcolor{black}{\cite[(2.14)]{HerMenRivII}}]
  \label{l:formal_Hecke_formula}
  For all~$n$ in ${\N \ssetminus p \N}$, $e$ and~$e'$ in~$\tSups$, and~$x$ in~$\hDss$,
  \begin{equation}
    \label{eq:87}
    T_n(\Piss(x))|_{\Dsspr}
    =
    \frac{1}{\# \Aut(\sspr)} \sum_{\phi \in \Hom_n(\ss,\sspr)} \Pisspr(\hphi\cdot x).
  \end{equation}
\end{lemma}

The natural action of $\Gss$ on $\bfR_{\ss', \ss}$, given by ${(g, \varphi) \mapsto g\circ\varphi}$, induces an action of the subgroup~$S_1(\ss, \ss)$ of~$\Gss$ on~$S_\ell(\ss', \ss)$.
This action is transitive and by isometries, so it possesses a unique invariant Borel probability measure~$\mu_\ell^{\ss, \ss'}$ and the support of this measure equals~$S_\ell(\ss', \ss)$ \cite[Lemma~5.3 and Proposition~6.2$(ii)$]{HerMenRivII}.

Let~$E$ be in~$\Sups$.
Then, $E$ is a \emph{formal \CM{} point} if $\End(\mathcal{F}_E)$ is isomorphic to an order in a quadratic extension of~$\Qp$.
Define the subgroup~$\NE$ of~$\Zp^{\times}$ as follows.
If~$E$ is not a formal \CM{} point, then ${\NE \= (\Zp^{\times})^2}$.
In the case where~$E$ is a formal \CM{} point, denote by~$\Aut(\FE)$ the group of automorphisms of~$\FE$ defined over~$\OQpalg$, and by~$\nr_E$ the norm map of the field of fractions of~$\End(\FE)$ to~$\Qp$.
Then,
\begin{displaymath}
  \NE
  \=
  \left\{ \nr_E \left( \varphi \right) \: \varphi \in \Aut(\FE) \right\}.
\end{displaymath}
In all cases, $\NE$ is a multiplicative subgroup of~$\Zp^{\times}$ containing~$(\Zp^{\times})^2$.
In particular, the index of~$\NE$ in~$\Zp^{\times}$ is at most two if~$p$ is odd, and at most four if ${p = 2}$.

As in Section~\ref{ss:Equidistribucion}, define, for each coset~$\coset$ in~$\Zp^{\times}/ \NE$, the partial Hecke orbit
\begin{equation}
  \label{eq:88}
  \corbit
  \=
  \bigcup_{n \in \coset \cap \N} \supp(T_n(E)).
\end{equation}
The closure~$\corbitc$ in~$\Sups$ of this set is compact \cite[Theorem~C]{HerMenRivII}.
Let~$\ss_0$ in~$\tSups$ be the image of~$E$ under the reduction map, and let~$\ss$ be in~$\tSups$.
For each~$x$ in~$\Pi_{\ss_0}^{-1}(E)$, evaluation at~$x$ defines a continuous map ${\Ev^{x, \ss} \: \bfG_{\ss_0, \ss} \rightarrow \hDss}$.
For each~$\ell$ in~$\coset$, the measure
\begin{equation}
  \label{eq:def_mu_E,e,coset}
  (\Pi_{\ss} \circ\Ev^{x, \ss})_*(\mu_\ell^{\ss_0, \ss})
\end{equation}
is independent of~$\ell$ and~$x$.
Denote it by~$\mu_{\coset}^{E, \ss}$ and put
\begin{equation}
  \label{eq:def_mu_coset}
  \mu_{\coset}^E
  \=
  \frac{24}{p-1} \sum_{\ss \in \tSups} \frac{1}{\# \Aut(\ss)} \mu_{\coset}^{E, \ss}.
\end{equation}
The support of~$\mu_{\coset}^E$ equals~$\corbitc$ \cite[Proposition~6.4]{HerMenRivII}, and the Eichler mass formula ensures that~$\mu_{\coset}^E$ is a Borel probability measure on~$\Ell(\Cp)$.
For all~$\ell$ in ${\N \ssetminus p\N}$, $E'$ in~$\OrbEc$, and~$E''$ in~$\overline{\Orb_{\ell\NE}(E)}$,
\begin{equation}
  \label{eq:89}
  \supp(T_\ell(E'))
  \subseteq
  \overline{\Orb_{\ell\NE}(E)},
  \supp(T_\ell(E''))
  \subseteq
  \OrbEc,
\end{equation}
\begin{equation}
  \label{eq:90}
  (\bT_\ell)_* \muE
  =
  \mu_{\ell\NE}^E,
  \text{ and }
  (\bT_\ell)_* \mu_{\ell\NE}^E
  =
  \muE,
\end{equation}
see \cite[Corollary~6.1]{HerMenRivII}.

The following equidistribution result is a special case of \cite[Theorem~C]{HerMenRivII}.

\begin{theorem}
  \label{t:equid_partial_hecke}
  Let~$E$ be in~$\Sups$ and~$\coset$ in~$\Zp^{\times}/ \NE$.
  Then, for every sequence $(n_j)_{j=1}^\infty$ in ${\coset \cap \N}$ tending to~$\infty$, the following weak convergence of measures holds,
  \begin{equation}
    \label{eq:91}
    \odelta_{T_{n_j}(E)} \to \mu_{\coset}^E \text{ as } j \to \infty.
  \end{equation}
\end{theorem}

Section~\ref{s:random_walks} relies on the following result.

\begin{theorem}[\textcolor{black}{\cite[Theorem~3.2]{HerMenRiv24}}]
  \label{t:no_atoms_hecke}
  For all~$E$ in~$\Sups$ and~$\coset$ in~$\Zp^{\times}/ \NE$, the measure~$\mu_{\coset}^E$ is nonatomic.
  That is, for every~$x$ in~$\Ell(\Cp)$, the equality ${\mu_{\coset}^E(\{x\}) = 0}$ holds.
\end{theorem}

\subsection{Spectral gap property on the supersingular locus}
\label{s:equid_in_sups_locus}

Let~$\delta$ be in~$\Rp$ and~$E$ in~$\Sups$.
Denote by~$\dconst$ be the real vector space of functions ${G \: \Ell(\Cp) \to \R}$ supported on~$\OrbEc$ such that, for every~$\ss$ in~$\tSups$, the function ${G \circ \Piss}$ is constant on the intersection of every ball of radius~$\delta$ of~$\hDss$ with~$\Piss^{-1}(\OrbEc)$.
Since~$\OrbEc$ is compact and ${\Piss \: \hDss \rightarrow \Dss}$ is a ramified covering of finite degree, the set~$\Piss^{-1}(\OrbEc)$ is compact and thus~$\dconst$ is finite dimensional.

\begin{lemma}
  \label{l:Tell_on_dconst}
  Let~$\delta$ be in~$\Rp$ and~$E$ in~$\Sups$.
  Then, for all~$\ell$ in ${\N \ssetminus p\N}$ and~$E'$ in~$\overline{\Orb_{\ell \NE}(E)}$, the inclusions ${\bT_\ell(\dconst) \subseteq W_{\delta,E'}}$ and ${\bT_\ell(W_{\delta,E'}) \subseteq \dconst}$ hold.
  Moreover, for all~$f$ in~$\dconst$ and~$\whf$ in~$W_{\delta,E'}$,
  \begin{equation}
    \label{eq:inv_measure_Tell}
    \int \bT_\ell f \dd \mu_{\ell\NE}^E
    =
    \int f \dd \muE
    \text{ and }
    \int \bT_\ell \whf \dd \muE
    =
    \int \whf \dd \mu_{\ell\NE}^E.
  \end{equation}
\end{lemma}

\begin{proof}
  The inclusions ${\bT_\ell(\dconst) \subseteq W_{\delta,E'}}$ and ${\bT_\ell(W_{\delta,E'}) \subseteq \dconst}$ are a direct consequence of Lemmas~\ref{l:action_isometries} and~\ref{l:formal_Hecke_formula} and~\eqref{eq:89}.
  Additionally, for all~$f$ in~$\dconst$ and~$\whf$ in~$W_{\delta, E'}$, \eqref{eq:90} implies
  \begin{equation}
    \label{eq:92}
    \int \bT_\ell f \dd \mu_{\ell \NE}^E
    =
    \int f \dd (\bT_\ell)_* \mu_{\ell\NE}^E
    =
    \int f \dd \muE
  \end{equation}
  and
  \begin{equation}
    \label{eq:93}
    \int \bT_\ell \whf \dd \muE
    =
    \int \whf \dd (\bT_\ell)_* \muE
    =
    \int \whf \dd \mu_{\ell \NE}^E.
    \qedhere
  \end{equation}
\end{proof}

Endow~$\dconst$ with the inner product~$\langle \cdot, \cdot \rangle_{\delta, E}$ defined by ${\langle G, \hG \rangle_{\delta, E} \= \int G \hG \dd \muE}$, and denote by~$\| \cdot \|_{\delta, E}$ the corresponding norm.
Let~$\bfone_{\delta, E}$ be the element of~$\dconst$ that is constant equal to~$1$ on~$\OrbEc$, $\zeromean$ the subspace of~$\dconst$ defined by
\begin{equation}
  \label{eq:94}
  \zeromean
  \=
  \left\{ G \in \dconst \: \int G \dd \muE = 0 \right\},
\end{equation}
${P_0 \: \dconst \to \zeromean}$ the orthogonal projection defined by
\begin{equation}
  \label{eq:95}
  P_0(G)
  \=
  G - \left(\int G \dd \muE\right) \bfone_{\delta, E},
\end{equation}
and~${\| \cdot \|_{\delta,E,0}}$ the seminorm defined by ${\| G \|_{\delta,E,0} \= \| P_0 (G) \|_{\delta,E}}$.

For each~$\ell$ in ${\NE \cap \N}$, Lemma~\ref{l:Tell_on_dconst} implies that~$\bT_\ell$ acts as a self-adjoint linear map~$U_{E, \ell}$ from~$\dconst$ into itself, satisfying~$U_{E, \ell}(\bfone_{\delta, E}) = \bfone_{\delta, E}$ and~$\| U_{E, \ell} \|_{\delta,E} \le 1$.
In particular, $U_{E, \ell}$ leaves~$\zeromean$ invariant and ${\| U_{E, \ell} \|_{\delta,E,0} \le 1}$.
Similarly, for every~$\ell$ in ${\N \ssetminus p \N}$ the map $\bT_\ell^2$ acts on~$\dconst$ as a self-adjoint linear map~$V_{E, \ell}$, satisfying ${V_{E, \ell}(\bfone_{\delta, E}) = \bfone_{\delta, E}}$ and ${\| V_{E, \ell} \|_{\delta,E} \le 1}$.
In particular, $V_{E, \ell}$ leaves~$\zeromean$ invariant and ${\| V_{E, \ell} \|_{\delta,E,0} \le 1}$.
If~$\ell$ belongs to~$\NE$, then ${V_{E, \ell} = U_{E, \ell}^2}$.

The proof of Theorem~\ref{t:Equidistribucion}$(ii)$ in Section~\ref{ss:Equidistribucion} relies on the following theorem.

\begin{theorem}
  \label{t:hecke-dynamics_sups}
  For all~$\delta$ in~$\Rp$, $E$ in~$\Sups$, and~$\ell$ in ${\N \ssetminus p \N}$, the following properties hold.
  \begin{enumerate}
  \item[$(i)$]
    \label{p:check-contraction}
    The inequality ${\| V_{E, \ell} \|_{\delta,E,0} < 1}$ holds.
    If in addition~$\ell$ belongs to~$\NE$, then ${\| U_{E, \ell} \|_{\delta,E,0} < 1}$.
  \item[$(ii)$]
    There exists a constant~$C_1$ in~$\Rp$, depending only on~$\delta$ and~$\OrbEc$, such that, for all~$E_0$ in~$\OrbEc$, $F$ in~$\dconst$, and~$n$ in~$\N$,
    \begin{equation*}
      \left|\int F \dd \odelta_{T_\ell^{2n}(E_0)} - \int F \dd \muE\right|
      \le
      C_1 \| V_{E, \ell} \|_{\delta,E,0}^n\| F \|_{\delta,E}.
    \end{equation*}
  \item[$(iii)$]
    There exists a constant~$C_1'$ in~$\Rp$, depending only on~$\delta$ and~$\overline{\Orb_{\ell\NE}(E)}$, such that, for all~$E_0$ in~$\OrbEc$, $E'$ in~$\overline{\Orb_{\ell \NE}(E)}$, ${F}$ in~$W_{\delta,E'}$, and~$n$ in~$\N$,
    \begin{equation*}
      \left|\int F \dd \odelta_{T_\ell^{2n + 1}(E_0)} - \int F \dd \mu_{\ell \NE}^E \right|
      \le
      C_1' \| V_{E, \ell} \|_{\delta,E,0}^n \| F \|_{\delta,E'}.
    \end{equation*}

  \end{enumerate}
\end{theorem}

\begin{proof}
  Fix~$\delta$ in~$\Rp$ and~$E$ in~$\Sups$.

  To prove item~$(i)$, choose, for each~$\coset$ in~$\Zp^{\times}/\NE$, a point~$E_\coset$ in~$\overline{\Orb_{\coset}(E)}$ and define~$W_0$ as the orthogonal sum $W_0 \= \bigoplus_{\coset \in \Zp^{\times}/\NE} W_{\delta,E_\coset, 0}$.
  Lemma~\ref{l:Tell_on_dconst} implies that~$(\bT_\ell)_{\ell \in \N \ssetminus p\N}$ is a commutative family of self-adjoint operators acting on the finite dimensional vector space~$W_0$.
  It follows that there exists an orthonormal basis~$\cB$ of~$W_0$ consisting of simultaneous eigenfunctions of each element of~$(\bT_\ell)_{\ell \in \N \ssetminus p\N}$.
  For every~$G$ in~$\cB$, Theorem~\ref{t:equid_partial_hecke} implies the convergence ${\bT_\ell G \to 0}$ as ${\ell \to \infty}$ with~$\ell$ in ${\N \ssetminus p\N}$.
  It follows that~$\|\bT_\ell\|_{W_0} \to 0$ as ${\ell \to \infty}$ with~$\ell$ in ${\N \ssetminus p\N}$.
  Combined with Proposition~\ref{p:criterion_convergence_iterates}$(ii)$, this yields, for every~$\ell$ in ${\N \ssetminus p\N}$, the inequality ${\|\bT_\ell\|_{W_0}<1}$.
  Since~$V_{E, \ell}$ is the restriction of~$\bT_\ell^2$ to $\zeromean$, it follows that $\|V_{E, \ell} \|_{\delta,E,0}<1$.
  If~$\ell$ belongs to~$\NE$, then~${U_{E, \ell}}$ is the restriction of~$\bT_\ell$ to $\zeromean$ and thus ${\| U_{E, \ell} \|_{\delta,E,0} < 1}$, as desired.

  To prove item~$(ii)$, note that the finite dimensionality of~$\dconst$ implies that there is~$C_1$ in~$\Rp$ such that, for every~$f$ in~$\dconst$,
  \begin{equation}
    \label{eq:96}
    \sup_{x \in \Ell(\Cp)} |f(x)|
    \le
    C_1 \|f\|_{\delta,E}.
  \end{equation}
  Then, for all~$E_0$ in~$\OrbEc$ and~$F$ in~$\dconst$,
  \begin{multline}
    \label{eq:97}
    \left| \int F \dd \odelta_{T_\ell^{2n}(E_0)} - \int F \dd \muE \right|
    =
    \left| (V_{E, \ell}^n F) (E_0) - \langle F, \bfone_{\delta, E} \rangle_{\delta, E} \right|
    =
    |V_{E, \ell}^n (P_0(F))(E_0)|
    \\ \le
    C_1 \|V_{E, \ell}^n (P_0(F)) \|_{\delta, E}
    \le
    C_1 \|V_{E, \ell} \|_{\delta, E, 0}^n \|P_0(F) \|_{\delta, E, 0}
    \le
    C_1 \|V_{E, \ell} \|_{\delta, E, 0}^n \|F\|_{\delta, E}.
  \end{multline}

  To prove item~$(iii)$, let~$E_0$ be in~$\OrbEc$, $E'$ in~$\overline{\Orb_{\ell \NE}(E)}$, $F$ in~$W_{\delta,E'}$, and~$C_1$ as above.
  Item~$(ii)$ combined with Lemma~\ref{l:Tell_on_dconst}, yields, for every~$n$ in~$\N$,
  \begin{multline}
    \label{eq:98}
    \left|\int F \dd \odelta_{T_\ell^{2n + 1}(E_0)} - \int F \dd \mu_{\ell \NE}^E\right|
    =
    \left|\int \bT_\ell F \dd \odelta_{T_\ell^{2n}(E_0)} - \int \bT_\ell F \dd \muE\right|
    \\ \le
    C_1 \| V_{E, \ell} \|_{\delta, E, 0}^n \|\bT_\ell F\|_{\delta,E}
    \le
    (C_1 \|\bT_\ell\|_{W_0}) \| V_{E, \ell} \|_{\delta, E, 0}^n \|F\|_{\delta,E'}.
  \end{multline}
  This proves item~$(iii)$ with ${C_2 = C_1 \|\bT_\ell\|_{W_0}}$ and completes the proof of the theorem.
\end{proof}

\begin{proof}[Proof of Theorem~\ref{t:Equidistribucion}$(ii)$]
  Let~$\varepsilon$ be in~$\Rp$ and choose~$E'$ in~$\overline{\Orb_{\ell \NE}(E)}$.
  Let~$\delta$ in~$\Rp$ be sufficiently small, so there are~$F$ in~$\dconst$ and~$\hF$ in~$W_{\delta,E'}$ satisfying
  \begin{equation}
    \label{eq:99}
    \sup_{\OrbEc} |f - F|
    \le
    \frac{\varepsilon}{3}
    \text{ and }
    \sup_{\overline{\Orb_{\ell \NE}(E)}} |f - \hF|
    \le
    \frac{\varepsilon}{3}.
  \end{equation}
  Let~$C_1$ and~$C_1'$ be given by items~$(ii)$ and~$(iii)$ of Theorem~\ref{t:hecke-dynamics_sups}, respectively, and let~$N$ in~$\N$ be sufficiently large so that ${C_1 \| V_{E, \ell} \|_{\delta,E,0}^N\| F \|_{\delta,E} \le \frac{\varepsilon}{3}}$ and ${C_1' \| V_{E, \ell} \|_{\delta,E,0}^N\| \hF \|_{\delta,E'} \le \frac{\varepsilon}{3}}$.
  By~\eqref{eq:89}, for all~$E_0$ in~$\OrbEc$ and~$n$ in~$\N$ satisfying ${n \ge N}$,
  \begin{multline}
    \label{eq:100}
    \left| \bT_\ell^{2n}(f)(E_0) - \int f \dd \muE \right|
    =
    \left| \int f \dd \odelta_{T_\ell^{2n}(E_0)} - \int f \dd \muE \right|
    \\ \le
    \left| \int F \dd \odelta_{T_\ell^{2n}(E_0)} - \int F \dd \muE \right| + \frac{2\varepsilon}{3}
    \le
    C_1 \| V_{E, \ell} \|_{\delta,E,0}^n\| F \|_{\delta,E} + \frac{2\varepsilon}{3}
    \le
    \varepsilon
  \end{multline}
  and
  \begin{multline}
    \label{eq:101}
    \left| \bT_\ell^{2n + 1}(f)(E_0) - \int f \dd \mu_{\ell \NE}^E \right|
    =
    \left| \int f \dd \odelta_{T_\ell^{2n + 1}(E_0)} - \int f \dd \mu_{\ell \NE}^E \right|
    \\ \le
    \left| \int \hF \dd \odelta_{T_\ell^{2n + 1}(E_0)} - \int \hF \dd \mu_{\ell \NE}^E \right| + \frac{2\varepsilon}{3}
    \le
    C_1' \| V_{E, \ell} \|_{\delta,E,0}^n\| \hF \|_{\delta,E'} + \frac{2\varepsilon}{3}
    \le
    \varepsilon.
  \end{multline}
  Since~$\varepsilon$ in~$\Rp$ is arbitrary, this completes the proof of the theorem.
\end{proof}

\section{Central limit theorems}
\label{s:TLC}

This section proves Theorem~\ref{t:TLC} in Section~\ref{ss:TLC} and its variants (Remark~\ref{rmk:case_ell_not_in_NE}).
After preliminaries on Markov chains, Section~\ref{t:CLT_Markov} recalls the (quenched) central limit theorem of Derriennic and Lin in~\cite{DL03}.
Section~\ref{Markov} introduces a Markov chain associated with~$T_\ell$ and establishes some lemmas.
The proof of Theorem~\ref{t:TLC} and its variants (Remark~\ref{rmk:case_ell_not_in_NE}) occupy Section~\ref{sec:proof_TLC}.
Besides~\cite{DL03}, it relies on the spectral gap property given by Theorem~\ref{t:hecke-dynamics_sups}$(i)$.

\subsection{Markov chains}
See, \emph{e.g.}, \cite[\emph{Chapitre}~V]{Nev64} for background on Markov chains.

Let $(X,\cB)$ be a measurable space.
A \emph{transition probability on~$X$} is a function ${P \: X \times \cB \to [0,1]}$ such that, for every~$A$ in~$\cB$, the map ${x \mapsto P(x,A)}$ is measurable, and, for every~$x$ in~$X$, the map $A\mapsto P(x,A)$ is a probability measure denoted by~$P_x$.
The \emph{Markov operator~$P$} acting on the space of measurable and bounded functions ${f \: X \to \R}$, is defined by
\begin{equation}
  \label{eq:102}
  Pf(x)
  \=
  \int f(y) \dd P_x(y).
\end{equation}
The same formula defines~$Pf$ for each measurable function~$f$ defined on~$X$ that takes nonnegative values and is not necessarily bounded.

Denote by~$\cB^{\otimes \Nz}$ the product $\sigma$-algebra on~$X^{\Nz}$ and, for every~$n$ in~$\Nz$, denote by ${X_n \: X^{\Nz} \to X}$ the $n$-th projection map.
Given a probability measure~$\nu$ on $(X,\cB)$, denote by~$\P_\nu$ the probability measure on $(X^{\Nz}, \cB^{\otimes \Nz})$ uniquely determined by the following property: For every~$A$ in~$\cB$,
\begin{equation}
  \label{eq:103}
  \P_\nu\left(A \times \prod_{i=1}^\infty X \right)
  =
  \nu(A)
\end{equation}
and, for every~$T$ in~$\N$ and every sequence~$(A_i)_{i=0}^\infty$ of measurable sets with ${A_i= X}$ whenever $i>T$,
$$
\P_\nu\left(\prod_{i=0}^\infty A_i\right)
=
\int_{A_0} \int_{A_1} \cdots \int_{A_T} \dd P_{x_{T-1}}(x_T) \cdots \dd P_{x_0}(x_1) \textbf{} \dd \nu(x_0),
$$
see, \emph{e.g.}, \cite[\emph{Chapitre}~V, \emph{Corollaire}~II]{Nev64}.
Thus, the distribution of~$X_0$ for~$\P_\nu$ is~$\nu$ and, for every~$n$ in~$\Nz$ and~$A$ in~$\cB$,
\begin{equation}
  \label{eq:104}
  \P_\nu[X_{n + 1} \in A \mid X_n = x]
  =
  P(x, A).
\end{equation}
That is, ${(X_n)_{n = 0}^\infty}$ is a Markov chain with initial distribution~$\nu$ and transition probability~$P$.
Denoting by~$\E_\nu$ the expectation with respect to~$\P_\nu$, the Markov property reads, for every~$n$ in~$\Nz$ and every bounded and measurable ${f \: X \to \R}$,
\begin{equation}
  \label{eq:105}
  \E_\nu [f(X_{n + 1})|X_n, \ldots, X_0]
  =
  Pf(X_n)
  \text{ almost everywhere for~$\P_\nu$.}
\end{equation}
The map $f\mapsto f(X_0)$ defines an isometric embedding $L^2(X,\nu) \to L^2(X^{\Nz},\P_\nu)$.
By the definition of~$\P_\nu$, for every~$T$ in~$\Nz$ and every nonnegative measurable function ${F \: X^{\Nz} \to \R}$ that depends only on the first $T+1$ coordinates~$(x_0,x_1,\ldots,x_T)$ of~$(x_0,x_1,\ldots)$ in~$X^{\Nz}$,
\begin{equation}
  \label{eq:int_F_dP_nu}
  \int F \dd \P_\nu
  \\ =
  \int \cdots \iint F(x_0, x_1, \ldots) \dd P_{x_{T-1}}(x_T) \cdots \dd P_{x_0}(x_1) \dd \nu(x_0).
\end{equation}

A measure~$\mu$ on $(X,\cB)$ is \emph{$P$-invariant} if, for every bounded and measurable function ${f \: X \to \R}$, the equality ${\int Pf \dd \mu = \int f \dd \mu}$ holds.
In this case, ${P}$ induces a contraction ${P \: L^2(X,\mu) \to L^2(X,\mu)}$.
The measure~$\mu$ is $P$-\emph{ergodic} if every bounded and measurable function ${f \: X \to \R}$ satisfying ${Pf=f}$ almost everywhere is constant almost everywhere.
In terms of the \emph{shift map ${\sigma \: X^{\Nz} \to X^{\Nz}}$}, defined by
\begin{equation}
  \label{eq:106}
  \sigma(x_0, x_1, \ldots)
  \=
  (x_1, \ldots),
\end{equation}
a measure~$\mu$ on $(X,\cB)$ is $P$-invariant if and only if~$\P_\mu$ is invariant for~$\sigma$, and~$\mu$ is $P$-ergodic if and only if~$\P_\mu$ is ergodic for~$\sigma$.

Let~$\mu$ be a probability measure on $(X,\cB)$ that is $P$-invariant and $P$-ergodic, let~$f$ be in~$L^2(X,\mu)$, and, for every~$n$ in~$\N$, define ${\hS_n(f) \: X^{\Nz} \to \R}$ as
\begin{equation}
  \label{eq:hatS_n(f)}
  \hS_n(f)
  \=
  \sum_{k=0}^{n - 1}f(X_k).
\end{equation}
The sequence of random variables $(f(X_n))_{n=0}^\infty$ satisfies the strong law of large numbers by the pointwise ergodic theorem of Birkhoff.
Namely, for $\P_\mu$-almost every~$\ux$ in~$X^{\Nz}$,
\begin{equation}
  \label{eq:107}
  \frac{1}{n} \hS_n(f)(\ux) \to \int f \dd \mu \text{ as } n \to \infty.
\end{equation}
When ${\int f \dd \mu = 0}$ and for a given probability measure~$\nu$ on~$(X, \cB)$, one can ask about the convergence of $\frac{1}{\sqrt{n}} \hS_n(f)$ in law under~$\P_\nu$.
That is, one can ask about the weak convergence of the sequence $\left( \left(\frac{1}{\sqrt{n}} \hS_n(f) \right)_*(\P_\nu) \right)_{n = 0}^\infty$ of Borel probability measures on $\R$.
Given~$\sigma$ in~$\R$ and~$f$ in~$L^2(X,\mu)$ with ${\sigma > 0}$ and ${\int f \dd \mu = 0}$, the \emph{central limit theorem under $\P_\nu$ with variance $\sigma^2$ holds for~$f$}, if
\begin{equation}
  \label{eq:108}
  \left(\frac{1}{\sqrt{n}} \hS_n(f) \right)_*(\P_\nu) \to \cN_{0,\sigma^2} \text{ weakly as } n \to \infty,
\end{equation}
where $\cN_{0,\sigma^2}$ is the centered Gauss normal distribution with variance~$\sigma^2$ defined by~\eqref{eq:Gauss_dist}.
Such a central limit theorem is known to hold for several types of functions~$f$.
See, \emph{e.g.}, \cite{DL01} and the references therein.

The following theorem is a special case of a result of Derriennic and Lin in~\cite{DL03}.
For each~$x$ in~$X$, denote by $\P_{x}$ the associated probability measure~$\P_\nu$ on $X^{\Nz}$ with ${\nu=\delta_{x}}$.

\begin{theorem}[\textcolor{black}{\cite[Theorem]{DL03}}]
  \label{t:CLT_Markov}
  Let $(X,\cB,\mu)$ be a probability space and~$P$ a transition probability on~$X$ such that~$\mu$ is $P$-invariant and $P$-ergodic.
  Let~$f$ in~$L^2(X,\mu)$ be such that ${\int f \dd \mu = 0}$ and $\sum_{k=0}^\infty P^kf$ converges in~$L^2(X,\mu)$.
  Then,
  \begin{equation}
    \label{eq:sigma_f_Markov}
    \int f^2\dd\mu +2\sum_{k=1}^\infty \int f \cdot P^kf \dd\mu
  \end{equation}
  converges to a nonnegative real number.
  If this number is strictly positive, then, for $\mu$-almost every~$x$ in~$X$, the central limit theorem under~$\P_x$ with variance~\eqref{eq:sigma_f_Markov} holds for $f$.
  Moreover, if~\eqref{eq:sigma_f_Markov} is zero, then, for $\mu$-almost every~$x$ in~$X$, the following weak convergence holds,
  \begin{equation}
    \label{eq:109}
    \left(\frac{1}{\sqrt{n}} \hS_n(f) \right)_*(\P_x) \to \delta_0 \text{ as } n \to \infty.
  \end{equation}
\end{theorem}

\begin{remark}
  \label{rmk:CLT_atomic_is_trivial}
  Suppose~$\mu$ is a Dirac measure on~$X$ and let~$f$ be in~$L^2(X,\mu)$ with ${\int f \dd \mu=0}$.
  Then, ${L^2(X,\mu) = \R}$, ${f = 0}$, and hence, for every measure~$\nu$ on $(X,\cB)$, the sequence $\left( \frac{1}{\sqrt{n}} \hS_n(f) \right)_{n = 1}^\infty$ converges in law under $\P_\nu$ to $\delta_0$.
\end{remark}

\subsection{A Markov chain associated with a Hecke correspondence}
\label{Markov}
Throughout the rest of this section, fix a prime number~$p$.
For each~$\ell$ in~$\N$, let~$P_\ell$ be transition probability on~$\Ell(\Cp)$ defined, for all~$E$ in~$\Ell(\Cp)$ and Borel subset~$A$ of~$\Ell(\Cp)$, by
\begin{equation}
  \label{eq:P_ell_def}
  P_\ell(E,A)
  \=
  \odelta_{T_\ell(E)}(A).
\end{equation}
For each~$E$ in~$\Ell(\Cp)$, denote by~$P_{\ell, E}$ the Borel probability measure on~$\Ell(\Cp)$ defined by ${P_{\ell, E}(A) \= P_{\ell}(E, A)}$.
For every ${f \: \Ell(\Cp) \to \R}$,
\begin{equation}
  \label{eq:P_ellf}
  P_\ell f (E)
  =
  \int_{\Ell(\Cp)} f \dd \odelta_{T_\ell(E)}
  =
  \bT_\ell f(E).
\end{equation}

The following result is a direct consequence of~\eqref{eq:inv_measure_Tell} in Lemma~\ref{l:Tell_on_dconst} and Theorem~\ref{t:hecke-dynamics_sups}.

\begin{lemma}
  \label{l:invariance_and_ergodicity}
  Let~$E$ be in~$\Sups$ and~$\ell$ an integer in~$\NE$ satisfying ${\ell \ge 2}$.
  Then, $P_\ell$ is a transition probability on~$\Ell(\Cp)$ and~$\muE$ is $P_\ell$\nobreakdash-invariant and $P_\ell$\nobreakdash-ergodic.
\end{lemma}

For the following lemma, recall that ${S_n(f) \: \Ell(\Cp)^n \to \R}$ and $\hS_n(f) \: \Ell(\Cp)^{\Nz} \to \R$ are defined in~\eqref{eq:S_n(f)} and~\eqref{eq:hatS_n(f)}, respectively.

\begin{lemma}
  \label{l:Markov_vs_O^n_T_ell}
  Let~$\ell$ and~$n$ be in~$\N$ and~${f \: \Ell(\Cp) \to \R}$ Borel measurable.
  Then, for every Borel probability measure~$\nu$ on $\Ell(\Cp)$ and every interval~$I$ of~$\R$,
  \begin{equation}
    \label{eq:110}
    \left(\frac{1}{\sqrt{n}} \hS_n(f) \right)_*(\P_{\nu})(I)
    =
    \int \odelta_{\fO_{T_\ell}^{n - 1}(E)} \left( \left\{ \uE \in \Ell(\Cp)^n \: \frac{1}{\sqrt{n}} S_n(f)(\uE) \in I \right\} \right) \dd \nu(E).
  \end{equation}
\end{lemma}

\begin{proof}
  Denote by ${\bfone_I \: \R \to \{0,1\}}$ the characteristic function of~$I$ and by ${G_n \: \Ell(\Cp) \to \R}$ the function defined by ${G_1 \= \bfone_I \circ f}$ and, for ${n \ge 2}$, by
  \begin{equation}
    \label{eq:111}
    G_n(E_0)
    \=
    \int \cdots \int \bfone_I\circ \left(\frac{1}{\sqrt{n}} \hS_n(f) \right) \dd P_{\ell, E_{n-2}}(E_{n-1}) \cdots \dd P_{\ell, E_0}(E_1).
  \end{equation}

  Since for every~$\uE$ in~$\Ell(\Cp)^{\Nz}$ the number $\bfone_I\circ \left(\frac{1}{\sqrt{n}} \hS_n(f) \right)(\uE)$ depends only on the first~$n$ coordinates of~$\uE$, \eqref{eq:int_F_dP_nu} implies
  \begin{equation}
    \label{eq:112}
    \left(\frac{1}{\sqrt{n}} \hS_n(f) \right)_*(\P_{\nu})(I)
    =
    \int \bfone_I \circ \left(\frac{1}{\sqrt{n}} \hS_n(f) \right) \dd \P_\nu
    =
    \int G_n \dd \nu.
  \end{equation}
  It is then enough to prove that, for every~$E$ in~$\Ell(\Cp)$,
  \begin{equation}
    \label{eq:equality_G_and_push_forward}
    G_n(E)
    =
    \odelta_{\fO_{T_\ell}^{n - 1}(E)} \left( \left\{ \uE \in \Ell(\Cp)^n \: \frac{1}{\sqrt{n}} S_n(f)(\uE) \in I \right\} \right).
  \end{equation}
  If ${n=1}$, then this follows from the definition of~$G_1$ and
  \begin{equation*}
    \odelta_{\fO_{T_\ell}^0(E)}(\{ E' \in \Ell(\Cp) \: S_1(f)(E') \in I \})
    =
    \bfone_I(f(E)).
  \end{equation*}
  Suppose ${n\ge 2}$.
  Combining~\eqref{eq:O_T_ell(E)^n_def} and~\eqref{eq:P_ellf}, yields
  \begin{multline}
    \label{eq:G_n(E)}
    G_n(E)
    =
    \frac{1}{\sigma(\ell)^n} \sum_{\uC \in I_{n,\ell}(E)} \bfone_I\circ \left(\frac{1}{\sqrt{n}}S_n(f) \right)(\uE(\uC))
    =
    \int \bfone_I \circ \frac{1}{\sqrt{n}}S_n(f) \dd \odelta_{\fO_{T_\ell}^{n - 1}(E)}
    \\ =
    \odelta_{\fO_{T_\ell}^{n - 1}(E)} \left( \left\{ \uE \in \Ell(\Cp)^n \: \frac{1}{\sqrt{n}} S_n(f)(\uE) \in I \right\} \right).
    \qedhere
  \end{multline}
\end{proof}

\begin{lemma}
  \label{l:G_n_in_dconst}
  Let~$E$ be in~$\Sups$, $\ell$ in ${\NE \cap \N}$, ${\delta}$ in~$\Rp$, and~$f$ in~$\dconst$.
  Then, for every~$n$ in~$\N$ and every interval~$I$ of~$\R$, the function ${\Ell(\Cp) \to \R}$ supported on~$\OrbEc$ and defined on this set by
  \begin{equation}
    \label{eq:113}
    E' \mapsto \odelta_{\fO_{T_\ell}^{n - 1}(E')} \left( \left\{ \uE \in \Ell(\Cp)^n \: \frac{1}{\sqrt{n}} S_n(f)(\uE) \in I \right\} \right),
  \end{equation}
  belongs to~$\dconst$.
\end{lemma}

\begin{proof}
  For every~$n$ in~$\N$ and every interval~$I$ of~$\R$, denote by ${G_{n, I} \: \Ell(\Cp) \to \R}$ the function supported on~$\OrbEc$ and defined on this set by~\eqref{eq:113}.
  For ${n = 1}$, every interval~$I$ of~$\R$, and every~$E'$ in~$\OrbEc$,
  \begin{equation}
    \label{eq:114}
    G_{1, I}(E')
    =
    \odelta_{\fO_{T_\ell}^0(E')} \left( \left\{ E'' \in \Ell(\Cp) \: f(E'') \in I \right\} \right)
    =
    \bfone_I(f(E'')),
  \end{equation}
  so~$G_{1, I}$ belongs to~$\dconst$.
  Let~$n$ in~$\N$ be such that, for every interval~$I$ of~$\R$, the function~$G_{n, I}$ belongs to~$\dconst$.
  Let~$\ss$ be in~$\tSups$, $B$ a ball of radius~$\delta$ of~$\hDss$, and~$c$ the common value of~$f$ on ${\Piss(B) \cap \OrbEc}$.
  Let~$I$ be an interval of~$\R$ and put ${J \= \frac{\sqrt{n + 1} I - c}{\sqrt{n}}}$.
  Then, for every~$E'$ in ${\Piss(B) \cap \OrbEc}$,
  \begin{equation}
    \label{eq:115}
    \begin{split}
      G_{n + 1, I}(E')
      & =
        \int \bfone_I \circ \frac{1}{\sqrt{n + 1}}S_{n + 1}(f) \dd \odelta_{\fO_{T_\ell}^{n}(E')}
      \\ & =
           \frac{1}{\sigma(\ell)^{n + 1}} \sum_{\uC \in I_{n + 1,\ell}(E')} \bfone_I\circ \left(\frac{1}{\sqrt{n + 1}}S_{n + 1}(f) \right)(\uE(\uC))
      \\ & =
           \frac{1}{\sigma(\ell)^{n + 1}} \sum_{\substack{\uC \in I_{n + 1,\ell}(E') \\ (E_0, \ldots E_{n + 1}) \= \uE(\uC)}} \bfone_J \circ \left(\frac{1}{\sqrt{n}} S_n(f) \right)(E_1, \ldots, E_{n + 1})
      \\ & =
           \frac{1}{\sigma(\ell)^{n + 1}} \sum_{C \le E' \text{ of order~$\ell$}} \sum_{\uC \in I_{n,\ell}(E'/C)} \bfone_J \circ \left(\frac{1}{\sqrt{n}} S_n(f) \right)(\uE(\uC))
      \\ & =
           \frac{1}{\sigma(\ell)} \sum_{C \le E' \text{ of order~$\ell$}} \int \bfone_J \circ \frac{1}{\sqrt{n}}S_n(f) \dd \odelta_{\fO_{T_\ell}^{n - 1}(E'/C)}
      \\ & =
           \bT_\ell(G_{n, J})(E').
    \end{split}
  \end{equation}
  Since by hypothesis~$G_{n, J}$ belongs to~$\dconst$, the function~$\bT_\ell(G_{n, J})$ also belongs to~$\dconst$ by Lemma~\ref{l:Tell_on_dconst} and this proves that~$G_{n + 1, I}$ is constant on ${\Piss(B) \cap \OrbEc}$.
  Since this holds for every ball of radius~$\delta$ of~$\hDss$, it follows that~$G_{n + 1, I}$ belongs to~$\dconst$.
  The lemma follows by induction on~$n$.
\end{proof}

\subsection{Proof of Theorem~\ref{t:TLC}}
\label{sec:proof_TLC}

Fix~$E$ in~$\Sups$, an integer~$\ell$ in~$\NE$ satisfying ${\ell \ge 2}$, and a function ${f \: \Ell(\Cp) \to \R}$ with ${\int f \dd \muE = 0}$ whose restriction to~$\OrbEc$ is a locally constant.
By Lemma~\ref{l:invariance_and_ergodicity}, the measure~$\muE$ is $P_\ell$\nobreakdash-invariant and $P_\ell$\nobreakdash-ergodic.
Modifying~$f$ outside~$\OrbEc$ if necessary, suppose~$f$ is supported on this set.
In particular, ${f}$ is Borel measurable.
Since~$\OrbEc$ is compact, there exists~$\delta$ in~$\Rp$ such that~$f$ is constant on every ball of~$\OrbEc$ of radius~$\delta$.
Then, $f$ belongs to~$\zeromean$ by~\eqref{eq:84}.
Moreover, the correspondence~$\bT_\ell$ acts on~$\zeromean$ as an operator~$U_{E, \ell}$ satisfying ${\|U_{E, \ell} \|_{\delta,E,0} < 1}$ by Theorem~\ref{t:hecke-dynamics_sups}$(i)$.
It follows that the series~$\sum_{k=0}^\infty \bT_\ell^k f$ converges in~$\zeromean$ and hence in~$L^2(\Ell(\Cp), \muE)$.
This proves that the hypotheses of Theorem~\ref{t:CLT_Markov} are satisfied with ${X = \Ell(\Cp)}$, ${\mu = \muE}$, and ${P = P_\ell}$.
In particular, the series~\eqref{eq:sigma_f_Markov}, which equals~\eqref{eq:sigma_f_p-adic} by~\eqref{eq:P_ellf}, converges to a nonnegative real number.
Denote its square root by~$\sigma_f$, as in the statement of the theorem.
If ${\sigma_f > 0}$, then Theorem~\ref{t:CLT_Markov} implies that, for every interval~$I$ of~$\R$ and $\muE$-almost every point~$E'$ in~$\OrbEc$,
\begin{equation}
  \label{eq:116}
  \left(\frac{1}{\sqrt{n}} \hS_n(f) \right)_*(\P_{E'})(I) \to \cN_{0,\sigma^2}(I) \text{ as } n \to \infty.
\end{equation}
Combined with Lemma~\ref{l:Markov_vs_O^n_T_ell} with $\nu=\delta_{E'}$, this implies
\begin{equation}
  \label{eq:117}
  \odelta_{\fO_{T_\ell}^{n - 1}(E')} \left( \left\{ \uE \in \Ell(\Cp)^n \: \frac{1}{\sqrt{n}} S_n(f)(\uE) \in I \right\} \right) \to \cN_{0,\sigma^2}(I) \text{ as } n \to \infty.
\end{equation}
Since ${\supp(\muE) = \OrbEc}$, Lemma~\ref{l:G_n_in_dconst} implies that this convergence holds for each point~$E'$ of~$\OrbEc$ and is uniform on this set.

The proof of Theorem~\ref{t:TLC} when ${\sigma_f = 0}$ is similar to the previous case and is omitted.

\begin{remark}
  \label{rmk:case_ell_not_in_NE}
  Let~$E$ be in~$\Sups$ and~$\ell$ an integer outside~$\NE$ satisfying ${\ell \ge 2}$ and ${p \nmid \ell}$.
  Then, central limit theorems analogous to Theorem~\ref{t:TLC} and Corollary~\ref{c:TLC} hold for each of the sequences~$(T_\ell^{2n}(E))_{n=0}^\infty$ and~$(T_\ell^{2n+1}(E))_{n=0}^\infty$.
  To state these results precisely, put
  \begin{equation}
    \label{eq:118}
    \coset(0)
    \=
    \NE
    \text{ and }
    \coset(1)
    \=
    \ell \NE,
  \end{equation}
  and define, for each~$n$ in~$\Nz$ and~$E'$ in~$\Ell(\Cp)$, the $0$-dimensional cycles in~$Z_0(\Ell(\Cp)^{n+1})$,
  \begin{equation*}
    \fO_{T_\ell}^{n, 1}(E')
    \=
    \sum_{\substack{\uC \in I_{2n + 1,\ell}(E') \\ (E'_0, \ldots, E'_{2n + 1}) \= \uE(\uC)}} (E'_1, E'_3, \ldots, E'_{2n + 1}),
  \end{equation*}
  ${\fO_{T_\ell}^{0, 0}(E') \= E'}$, and, if ${n \ge 1}$,
  \begin{equation*}
    \fO_{T_\ell}^{n, 0}(E')
    \=
    \sum_{\substack{\uC \in I_{2n,\ell}(E') \\ (E'_0, \ldots, E'_{2n}) \= \uE(\uC)}} (E'_0, E'_2, \ldots, E'_{2n}).
  \end{equation*}
  Then, for all~$i$ in~$\{0, 1\}$ and~${f \: \Ell(\Cp) \to \R}$,
  \begin{equation}
    \label{eq:119}
    \int S_n(f) \dd \odelta_{\fO_{T_\ell}^{n - 1, i}(E')}
    =
    \sum_{k=0}^{n - 1} \bT_\ell^{2k+i}f(E').
  \end{equation}
  A straightforward adaptation of Theorem~\ref{t:TLC} yields the following result: For all~$i$ in~$\{0, 1\}$, $E'$ in~$\OrbEc$, and function ${f \: \Ell(\Cp) \to \R}$ with ${\int f \dd \mu_{\coset(i)}^E = 0}$ whose restriction to~$\overline{\Orb_{\coset(i)}(E)}$ is locally constant,
  \begin{equation}
    \label{eq:120}
    \int f^2 \dd \mu^E_{\coset(i)} + 2 \sum_{k=1}^\infty \int f \cdot \bT_\ell^{2k} f \dd \mu^E_{\coset(i)}
  \end{equation}
  converges to a nonnegative real number.
  Denote its square root by~$\sigma_{f, i}$.
  If ${\sigma_{f, i} > 0}$, then
  \begin{equation}
    \label{eq:121}
    \odelta_{\fO_{T_\ell}^{n - 1, i}(E')} \left( \left\{ \uE \in \Ell(\Cp)^n \: \frac{1}{\sqrt{n}}S_n(f)(\uE) \in I \right\} \right) \to \cN_{0,\sigma_{f, i}^2}(I) \text{ as } n\to \infty,
  \end{equation}
  and the convergence is uniform on~$E'$ in~$\overline{\Orb_{\coset(i)}(E)}$ for fixed~$I$.
  If ${\sigma_{f, i} = 0}$, then the same properties hold with~$\cN_{0,\sigma_{f, i}^2}$ replaced by~$\delta_0$.
  The corresponding analogues of Corollary~\ref{c:TLC} also hold.
\end{remark}

\section{Random walks and stationary measures}
\label{s:random_walks}

This section proves Theorem~\ref{t:random_walks} in Section~\ref{ss:random_walks}, by combining the work of Benoist and Quint in~\cite{BQ14} with results on partial Hecke orbits and their homogeneous measures from~\cite{HerMenRivII}.

Throughout this section, fix a prime number~$p$, an integer~$\ell$ satisfying ${\ell \ge 2}$ that is not divisible by~$p$, and a supersingular elliptic curve~$\ss$ over~$\Fpalg$.
Recall from Section~\ref{ss:random_walks} that ${\langle \ell \rangle = \{ \ell^n \: n \in \N\}}$, that~$\ellc$ denotes closure of~$\langle \ell \rangle$ in~$\Zp$, and that~$\Hss(\ell)$ is the closed subgroup of~$\Gss$ given by ${\Hss(\ell) = \{ g \in \Gss \: \nr(g) \in \ellc\}}$.
Furthermore, there is an embedding ${\End(\Fss) \to M_2(\Z_{p^2})}$ that restricts to a group morphism ${\Gss \to \GL_2(\Z_{p^2})}$, hence inducing a linear action of~$\Hss(\ell)$ on~$\Cp^2$.
This linear action induces the action of~$\Hss(\ell)$ on~$\P^1(\Cp)$ by M{\"o}bius transformations.

The proof of Theorem~\ref{t:random_walks} relies on the following consequence of the results of Benoist and Quint in~\cite{BQ14}.

\begin{proposition}
  \label{p:freely-equivariant}
  Let~$\mu$ be a Borel probability measure on~$\Hss(\ell)$ whose support generates~$\Hss(\ell)$ as a closed subsemigroup of~$\Gss$.
  Then, for every~$z$ in~$\P^1(\Cp)$, the following properties hold.
  \begin{enumerate}
  \item
    There exists a unique Borel probability measure~$\xi_z$ supported on ${\Hss(\ell) \cdot z}$ that is $\mu$\nobreakdash-stationary, and this measure is $\mu$-ergodic.
  \item
    For every~$z'$ in~$\Hss(\ell) \cdot z$ and $\mu^{\otimes \N}$-almost every sequence~$(b_n)_{n=1}^\infty$ of elements of~$\Hss(\ell)$, the following weak convergence of measures hold,
    \begin{equation}
      \label{eq:122}
      \frac{1}{n} \sum_{k = 1}^n \mu^{\ast k} \ast \delta_{z'} \to \xi_z
      \text{ and }
      \frac{1}{n} \sum_{k=1}^n\delta_{b_k \cdots b_1 z'} \to \xi_z
      \text{ as }
      n \to \infty.
    \end{equation}
  \end{enumerate}
\end{proposition}

The proof of Proposition~\ref{p:freely-equivariant} relies on a couple of lemmas.
The following one extends~\cite[Corollary~5.8.1]{GorKas} and verifies one of the hypothesis required to apply the work of Benoist and Quint~\cite{BQ14}.
Namely, that the linear action of~$\Hss(\ell)$ on~$\Cp^2$ is \emph{strongly irreducible}: The only $\Hss(\ell)$-invariant finite union of vector subspaces of~$\Cp^2$ are~$\{0\}$ and~$\Cp^2$.
Strong irreducibility implies \emph{semisimplicity}: Every $\Hss(\ell)$-invariant vector subspace of~$\Cp^2$ admits a $\Hss(\ell)$-invariant complementary subspace.

\begin{lemma}
  \label{l:strongly_irred_action}
  Every finite $\Hss(\ell)$-invariant subset of~$\P^1(\Cp)$ is empty.
  In particular, the linear action of~$\Hss(\ell)$ on~$\Cp^2$ is strongly irreducible and hence semisimple.
\end{lemma}

\begin{proof}
  Suppose there were a $\Hss(\ell)$-invariant finite subset~$\cA_0$ of~$\P^1(\Cp)$ that is nonempty.
  Fix~$a$ in~$\cA_0$ and put
  \begin{equation}
    \label{eq:123}
    \cA
    \=
    \Hss(\ell) \cdot a,
    \hcX
    \=
    \Phiss^{-1}(\cA),
    \text{ and }
    \cX
    \=
    \Piss(\hcX).
  \end{equation}
  The set~$\cA$ is contained in~$\cA_0$ and is hence finite.
  Since~$\Phiss$ is equivariant under the action of~$\Gss$, rigid analytic, and nonconstant, the subset~$\hcX$ of~$\Dss$ is $\Hss(\ell)$-invariant and discrete.
  Together with the fact~$\Piss$ is rigid analytic of finite degree, this implies that~$\cX$ is discrete.
  Choose~$E$ in~$\cX$.
  Combining Lemma~\ref{l:formal_Hecke_formula} with the fact that~$\cX$ is $\Hss(\ell)$-invariant, it follows that, for every~$n$ in~$\Nz$, the inclusion ${\supp(T_{\ell^{2n}}(E)|_{\Dss}) \subseteq \cX}$ holds.
  Combined with \eqref{eq:sin-tasa_bis} in Corollary~\ref{c:Equidistribucion}, this implies ${\supp(\muE|_{\Dss}) \subseteq \cX}$ and hence that~$\muE|_{\Dss}$ is either zero or purely atomic.
  By the definition~\eqref{eq:def_mu_coset} of~$\muE$, the restriction~$\muE|_{\Dss}$ is proportional to the probability measure~$\mu_{\NE}^{E, e}$ and hence nonzero and purely atomic.
  This contradicts the fact that~$\mu^E_{\NE}$ is nonatomic (Theorem~\ref{t:no_atoms_hecke}) and proves that every finite $\Hss(\ell)$-invariant subset of~$\P^1(\Cp)$ is empty.

  To prove the last assertion, suppose that the linear action of~$\Hss(\ell)$ on~$\Cp^2$ were not strongly irreducible.
  Then, there would be a finite collection~$\sL$ of $1$-dimensional $\Cp$-subspaces of~$\Cp^2$ such that~$\bigcup_{L \in \sL} L$ is $\Hss(\ell)$-invariant.
  For each~$L$ in~$\sL$, denote by~$x(L)$ the corresponding point of~$\P^1(\Cp)$.
  Then, the finite subset~$\{x(L) \: L \in \sL\}$ of~$\P^1(\Cp)$ would be $\Hss(\ell)$-invariant, which is absurd.
  This proves that the linear action of~$\Hss(\ell)$ on~$\Cp^2$ is strongly irreducible and hence semisimple.
\end{proof}

In the next lemma, identify~$\P^1(\Cp)$ with ${\Cp \cup \{ \infty \}}$ and define the degree of~$\infty$ over~$\Q_{p^2}$ as~$1$.

\begin{lemma}
  \label{l:freely-equivariant}
  Let~$z$ and~$z'$ in~$\P^1(\Cp)$ be of degree at least~$3$ over~$\Q_{p^2}$.
  Then, there exists an equivariant homeomorphism from~$\Hss(\ell) \cdot z$ to~$\Hss(\ell) \cdot z'$.
\end{lemma}

\begin{proof}
  For each~$g$ in~$\Hss(\ell)$, denote by~$\psi_g$ the corresponding M{\"o}bius transformation, and put ${\cH \= \{ \psi_g \: g \in \Hss(\ell) \}}$.
  Then, ${\cH}$ is a compact subgroup of~$\PGL_2(\Z_{p^2})$.
  For each~$w$ in~$\P^1(\C_p)$, denote by ${\ev_{w} \: \cH \to \P^1(\C_p)}$ the evaluation map given by ${\ev_{w}(\psi) \= \psi(w)}$.

  Every element of~$\cH$ distinct from the identity has at most~$2$ fixed points, each of which is of degree at most~$2$ over~$\Q_{p^2}$.
  Since the degree of~$z$ over~$\Q_{p^2}$ is at least~$3$ by hypothesis, it follows that the stabilizer of~$z$ in~$\cH$ is trivial.
  Thus, ${\ev_z}$ is injective.
  Since~$\cH$ is compact and ${\ev_z(\cH) = \Hss(\ell) \cdot z}$, it follows that~$\ev_z$ is a homeomorphism from~$\cH$ to~$\Hss(\ell) \cdot z$.
  Similarly, ${\ev_{z'}}$ is a homeomorphism from~$\cH$ to~$\Hss(\ell) \cdot z'$.
  Denote by ${h \: \Hss(\ell) \cdot z \to \Hss(\ell) \cdot z'}$ the homeomorphism defined by ${h \= \ev_{z'} \circ \ev_z^{-1}}$.
  To prove that~$h$ is equivariant, let~$g$ in~$\Hss(\ell)$ and~$z_0$ in ${\Hss(\ell) \cdot z}$ be given, and let~$g_0$ in~$\Hss(\ell)$ be such that ${g_0 \cdot z = z_0}$.
  Then,
  \begin{equation}
    \label{eq:124}
    \ev_z^{-1}(z_0)
    =
    \ev_z^{-1}(g_0 \cdot z)
    =
    \psi_{g_0},
    h(z_0)
    =
    \ev_{z'}(\psi_{g_0})
    =
    g_0 \cdot z',
  \end{equation}
  \begin{equation}
    \label{eq:125}
    \ev_z^{-1}(g \cdot z_0)
    =
    \ev_z^{-1}((gg_0) \cdot z)
    =
    \psi_{gg_0},
  \end{equation}
  and
  \begin{equation}
    \label{eq:126}
    h(g \cdot z_0)
    =
    \ev_{z'}(\psi_{gg_0})
    =
    (gg_0) \cdot z'
    =
    g \cdot (g_0 \cdot z')
    =
    g \cdot h(z_0).
    \qedhere
  \end{equation}
\end{proof}

The proof of Proposition~\ref{p:freely-equivariant} relies on the following terminology.
A nonempty closed subset~$\cA$ of~$\P^1(\Cp)$ is \emph{$\Hss(\ell)$-invariant} if, for every~$g$ in~$\Hss(\ell)$, it satisfies~${g(\cA) = \cA}$, and it is \emph{$\Hss(\ell)$\nobreakdash-minimal} if it is $\Hss(\ell)$-invariant and does not contain a proper $\Hss(\ell)$-invariant closed subset.

\begin{proof}[Proof of Proposition~\ref{p:freely-equivariant}]
  Put ${z_0 \= z}$ if~$z$ belongs to~$\P^1(\Qpalg)$.
  Otherwise, choose~$z_0$ in~$\P^1(\Qpalg)$ be of degree at least~$3$ over~$\Q_{p^2}$.
  In all of the cases, there is an equivariant homeomorphism ${h \: \Hss(\ell) \cdot z \to \Hss(\ell) \cdot z_0}$ (Lemma~\ref{l:freely-equivariant}).
  Fix a finite field extension~$\cK_0$ of~$\Q_{p^2}$ inside~$\Cp$ such that~$\P^1(\cK_0)$ contains~$z_0$.

  By~\cite[Theorem~1.5]{BQ14} and Lemma~\ref{l:strongly_irred_action}, the map ${\nu \mapsto \supp(\nu)}$ defines a bijection between
  \begin{equation}
    \label{eq:127}
    \{\mu\text{-ergodic probability measures on } \P^1(\cK_0) \}
    \text{ and }
    \{\Hss(\ell) \text{-minimal subsets of } \P^1(\cK_0) \}.
  \end{equation}
  Hence, the $\Hss(\ell)$-minimal set~$\Hss(\ell) \cdot z_0$ supports a unique $\mu$-ergodic measure~$\xi_{z_0}$.
  It follows that~$\xi_{z_0}$ is the unique $\mu$-stationary measure supported on the compact set~$\Hss(\ell) \cdot z_0$, see, \emph{e.g.}, \cite[Lemma~2.10]{BQ16book}.
  Since there is an equivariant homeomorphism from~$\Hss(\ell) \cdot z$ to~$\Hss(\ell) \cdot z_0$, it follows that~$\Hss(\ell) \cdot z$ supports a unique $\mu$-stationary measure~$\xi_z$ and this measure is $\mu$-ergodic.
  This proves item~$(i)$.

  To prove item~$(ii)$, let~$z'$ be in~$\Hss(\ell) \cdot z$.
  Combining Lemma~\ref{l:strongly_irred_action} with~\cite[Theorems~1.1(i) and~1.3]{BQ14} yields that, for $\mu^{\otimes \N}$-almost every sequence~$(b_n)_{n=1}^\infty$ of elements of~$\Hss(\ell)$, the limits in the weak topology
  \begin{equation}
    \label{eq:128}
    \lim_{n\to \infty} \frac{1}{n} \sum_{k = 1}^n \mu^{\ast k} \ast \delta_{h(z')}
    \text{ and }
    \lim_{n\to \infty} \frac{1}{n} \sum_{k = 1}^n \delta_{b_k \cdots b_1 h(z')}
  \end{equation}
  exist and are $\mu$-stationary.
  It follows that the limits
  \begin{equation}
    \label{eq:129}
    \lim_{n\to \infty} \frac{1}{n} \sum_{k = 1}^n \mu^{\ast k} \ast \delta_{z'}
    \text{ and }
    \lim_{n\to \infty} \frac{1}{n} \sum_{k = 1}^n \delta_{b_k \cdots b_1 z'}
  \end{equation}
  also exist and are $\mu$-stationary.
  Clearly, each of these measures is supported on~$\Hss(\ell) \cdot z'$, which equals~$\Hss(\ell) \cdot z$.
  By item~$(i)$ each of these measures equals~$\xi_z$.
  This completes the proof of item~$(ii)$ and of the proposition.
\end{proof}

\begin{proof}[Proof of Theorem~\ref{t:random_walks}]
  Note that
  \begin{equation}
    \label{eq:130}
    \ellc
    \subseteq
    (\Zp^{\times})^2 \cup \ell (\Zp^{\times})^2
    \text{ and }
    \Zp^{\times} \Hss(\ell)
    =
    \Zp^{\times}(S_1(\ss,\ss) \cup S_\ell(\ss,\ss)).
  \end{equation}

  Put
  \begin{equation}
    \label{eq:131}
    \hnu_x
    \=
    (\Ev^{x,\ss})_*\left(\frac{\mu_{1}^{\ss,\ss}+\mu_\ell^{\ss,\ss}}{2} \right),
  \end{equation}
  so ${\nu_x = (\Phiss)_* \hnu_x}$.
  The equivariance of~$\Phiss$ together with~\eqref{eq:86} imply, for every~$g$ in~$\Hss(\ell)$,
  \begin{equation}
    \label{eq:132}
    g_*\nu_x
    =
    (\Phiss\circ g)_* \hnu_x
    =
    \left(\Phiss\circ \Ev^{x,\ss} \right)_*\left(\frac{\mu_{\nr(g)}^{\ss,\ss}+\mu_{\ell \nr(g)}^{\ss,\ss}}{2} \right).
  \end{equation}
  Combined with the inclusion in~\eqref{eq:130} and the fact that the subgroup~$\Zp^{\times}$ of~$\Gss$ acts trivially on~$\hDss$, this implies
  \begin{equation}
    \label{eq:133}
    (\Ev^{x,\ss})_*(\mu_{\nr(g)}^{\ss,\ss}+\mu_{\ell \nr(g)}^{\ss,\ss})
    =
    (\Ev^{x,\ss})_*(\mu_{1}^{\ss,\ss}+\mu_{\ell}^{\ss,\ss})
    \text{ and }
    g_* \nu_x
    =
    \nu_x.
  \end{equation}
  This proves that~$\nu_x$ is~$\Hss(\ell)$-invariant.
  Using that~$\Zp^{\times}$ acts trivially on~$\hDss$ and hence on~$\P^1(\Cp)$, the equality in~\eqref{eq:130} implies
  \begin{multline}
    \label{eq:134}
    \supp(\nu_x)
    =
    \Phiss((\supp(\mu_{1}^{\ss,\ss}) \cup \supp(\mu_{\ell}^{\ss,\ss})) \cdot x)
    =
    \Phiss((S_1(\ss,\ss) \cup S_\ell(\ss,\ss)) \cdot x)
    \\ =
    (S_1(\ss,\ss) \cup S_\ell(\ss,\ss)) \cdot \Phiss(x)
    =
    \Hss(\ell) \cdot \Phiss(x).
  \end{multline}
  This proves the first equality in~\eqref{eq:20}.
  Moreover, by~\eqref{eq:def_mu_E,e,coset} and~\eqref{eq:def_mu_coset},
  \begin{equation}
    \label{eq:135}
    (\Piss)_* \hnu_x
    =
    \frac{\mu^{E,e}_{\NE}+\mu^{E,e}_{\ell\NE}}{2}
    \text{ and }
    \supp((\Piss)_* \hnu_x)
    =
    \left(\OrbEc \cup \overline{\Orb_{\ell\NE}(E)} \right) \cap \Dss.
  \end{equation}
  This implies the second equality in~\eqref{eq:20}.

  Let~$\mu$ be a Borel probability measure on~$\Hss(\ell)$ whose support generates~$\Hss(\ell)$ as a closed subsemigroup of~$\Gss$.
  The measure~$\nu_x$ is~$\mu$-stationary because it is $\Hss(\ell)$-invariant.
  By Proposition~\ref{p:freely-equivariant}$(i)$, ${\nu_x = \xi_{\Phiss(x)}}$ and thus~$\nu_x$ is $\mu$-ergodic.
  The remaining assertions of item~$(ii)$ follow by combining Proposition~\ref{p:freely-equivariant}$(ii)$ and ${\xi_{\Phiss(x)} = \nu_x}$.
  To complete the proof of item~$(i)$, let~$x'$ in~$\hDss$ be such that~$\Phiss(x')$ belongs to~$\Hss(\ell) \cdot \Phiss(x)$.
  Then, ${\nu_{x'}}$ is $\mu$-ergodic and~\eqref{eq:20} yields
  \begin{equation}
    \label{eq:136}
    \supp(\nu_{x'})
    =
    \Hss(\ell) \cdot \Phiss(x')
    =
    \Hss(\ell) \cdot \Phiss(x).
  \end{equation}
  Together with Proposition~\ref{p:freely-equivariant}$(i)$, this implies ${\nu_{x'} = \xi_{\Phiss(x)} = \nu_x}$ and completes the proof of item~$(i)$.
\end{proof}

\bibliographystyle{alpha}
\bibliography{papers}

\end{document}